\documentclass[11pt,reqno,
]{amsart}

\usepackage{
amssymb, graphicx, color, cite 
}

\usepackage{changes}
\definechangesauthor[name=Nikolas, color=blue]{N}
\definechangesauthor[name=Plamen, color=purple]{P}

\let\oldtocsection=\tocsection

\let\oldtocsubsection=\tocsubsection

\let\oldtocsubsubsection=\tocsubsubsection

\renewcommand{\tocsection}[2]{\hspace{0em}\oldtocsection{#1}{#2}}
\renewcommand{\tocsubsection}[2]{\hspace{1em}\oldtocsubsection{#1}{#2}}
\renewcommand{\tocsubsubsection}[2]{\hspace{2em}\oldtocsubsubsection{#1}{#2}}

\usepackage[shortlabels]{enumitem}
\usepackage{mathtools}
\mathtoolsset{showonlyrefs}
\usepackage[linkcolor=blue, citecolor=black, breaklinks=true, colorlinks=true, linkcolor=blue, citecolor=blue, urlcolor=blue]{hyperref}

\renewcommand{\L}{\mathcal{L}}

\date{\today}

\newtheorem{theorem}{Theorem}[section]

\newtheorem{lemma}{Lemma}[section]
\newtheorem{definition}{Definition}[section]
\newtheorem{corollary}{Corollary}[section]

\theoremstyle{definition}
\newtheorem{remark}{Remark}[section]

\DeclareMathOperator{\sgn}{sgn}
\newcommand{\X}{\mathsf{X}}

\DeclareMathOperator{\WF}{WF}

\newcommand{\R}{{\mathbb R}}

\newcommand{\Id}{\text{\rm Id}}
\renewcommand{\r}[1]{\eqref{#1}}
\newcommand{\PDO}{$\Psi$DO}
\newcommand{\be}[1]{\be\label{#1}}
\newcommand{\ee}{\end{equation}}
\renewcommand{\d}{\mathrm{d}}

\renewcommand{\i}{\mathrm{i}}

\newcommand{\bo}{\partial M}
\newcommand{\p}{\partial}

\newtheorem{example}{Example}[section]
\newcommand{\set}{\setminus 0}

\title[Scattering Rigidity for Hamiltonian systems]{Scattering Rigidity for Hamiltonian systems\\ with an application to Finsler  geometry
}

\author[N. Eptaminitakis] {Nikolas Eptaminitakis}
\address{Institut für Differentialgeometrie, Leibniz Universit\"at Hannover,
	Welfengarten 1, 30167 Hannover, Germany}

\author[P. Stefanov]{Plamen Stefanov}
\address{Department of Mathematics, Purdue University, West Lafayette, IN 47907}
\thanks{P.S. partly supported by  NSF  Grants DMS-2154489 and DMS-2452757. Part of this research was carried out while N.E. was visiting Purdue University, supported by the former grant.}

\begin{document}

\begin{abstract}
We study scattering rigidity for Hamiltonian systems on $T^*M\setminus 0$, where $M$ is a manifold with boundary equipped with a positively homogeneous Hamiltonian function $H(x,\xi)$. We show that $H$ can be uniquely determined by the scattering relation up to a canonical transformation fixing the boundary (in a suitable sense) for positive energy levels $H=E>0$. We define the travel times $T(x,y)$ between boundary points, and show that their linearization leads to an X-ray transform over Hamiltonian curves, which we invert. When $E=0$, scattering rigidity can be formulated in terms of a diffeomorphism of the zero energy surfaces which preserves the boundary and respects the orbits of the Hamiltonian flows there, as well as the restricted symplectic form. The travel times are replaced by a defining function of pairs of boundary points which can be connected by a locally unique zero bicharacteristic. Its linearization leads to the ``Hamiltonian light ray transform'' which we invert modulo a gauge as well. As an application of this phase space approach, we prove semiglobal lens rigidity of non-trapping Finsler manifolds. The group of the gauge transformations consists of certain canonical transformations composed with Legendre transforms.
\end{abstract}
\maketitle
 \tableofcontents

\section{Introduction}

This paper studies lens and boundary rigidity problems for  Hamiltonian systems on $T^*M\set $  related to Hamiltonians $H(x,\xi)$ homogeneous of order two, where $M$ is a compact manifold  with boundary, as well as their applications to Finsler  geometry. We are in part motivated by the study of hyperbolic systems, which, microlocalized under a non-degeneracy assumption, diagonalize to a collection of scalar equations with such Hamiltonians. The main question is to what extent the scattering relation $S$ or the travel times between boundary points determine $H$. Note that this is a formally \textit{underdetermined} problem: the data depends on $2n-2$ variables, accounting for the homogeneity, while $H$ depends on $2n-1$ variables. We give a complete characterization for the problem with local data.  The order of homogeneity can be any positive integer, actually, but we chose order two in  the exposition for two reasons: first, to be able to compare the results to the metric case $H=\frac12 g^{ij}(x)\xi_i\xi_j$, where $g$ is Riemannian or more generally, pseudo-Riemannian; and second, to make the results easier to apply to second order hyperbolic systems, in particular to anisotropic elasticity. 
The scattering relation and the travel times are encoded in the (singularities of) the Cauchy data for the corresponding PDE, see Appendix~\ref{sec_App}. 

The first round of results say, roughly speaking, that two Hamiltonians with the same data at  a fixed positive energy level (and therefore, at any positive one),  are related by the pull-back of a homogeneous canonical transformation $\kappa$ which, in a certain sense made precise in  Theorem~\ref{thm_H} below, is identity on the boundary. The proof is relatively straightforward. Note that this does not solve the Riemannian case  right away, which has been the subject of extensive research, see, e.g., \cite{SUV_anisotropic,S-AH,SUV_localrigidity, SU-JAMS, SU-lens, PestovU, LassasSU}. The Riemannian rigidity actually requires $\kappa$ to be generated by a diffeomorphism $\psi$ of the base variable, lifted to $TM$ (which we can identify with $T^*M$ when $g$ is fixed). This is a much harder problem. 

Under a non-conjugacy assumption, we define the travel times $T(x,y)$ between boundary points and then we linearize them with respect to $H$.  
The standard derivation in the Riemannian case relies on the first variation of the energy functional. We prove a Hamiltonian counterpart in Lemma~\ref{lemma_V} and, perhaps non-intuitively, show that a Hamiltonian curve connecting a fixed pair of points is not a critical point of the Hamiltonian version  of the energy functional. Despite that, the linearization of $T$ turns out to be an X-ray transform  on $T^*M$ of functions $f(x,\xi)$ along the Hamiltonian curves of $H$, which we denote by $\X f$.
The inversion of $\X$ up to a gauge is obvious: it is the derivative along the flow of a function $\phi(x,\xi)$ vanishing for $x\in\bo$.  Back to the Riemannian case, in linearized boundary rigidity, $f(x,\xi)$ is induced by a symmetric order two tensor and locally takes the form $f=f^{ij}(x)\xi_i\xi_j$; one would like to show that $\phi(x,\xi)$ is a one-form (linear in $\xi$) on $M$. Again, this is a harder problem with a lot of progress recently, see, e.g., \cite{ SUV-tensors, UV:local,MonardSU14,SU-caustics,SU-Duke, Gabriel_G_M_book, Sh-book}.

Our second set of results concerns Hamiltonian dynamics at zero energy level, motivated by propagation of singularities for pseudo-differential operators  (\PDO s) of real principal type. An important example is $H=\frac12 g^{ij}(x)\xi_i\xi_j$ with $g$ pseudo-Riemannian \cite{Ilmavirta-pseudo, S-Agrawal-LR} (in particular Lorentz \cite{LOSU-Light_Ray,S-support2014,S-Lorentzian-scattering,S-Lorentz_analytic,Yiran-Light-2021,Vasy-Wang-21}), but not Riemannian. An important distinction now is that $H^{-1}(\{0\})$ is conic, codimension one, and there is no proper way to define travel times. We show that one can recover  the null cone, up to a diffeomorphism preserving the orbits of the Hamiltonian flow and the restricted symplectic forms. 
The scattering relation is homogeneous, and it is the canonical relation of the corresponding DN (or DD) map, see section~\ref{sec_A2} again. 
We characterize the kernel of the associated ``light ray transform'' $L$ as well. This answers questions posed in \cite{OSSU-principal} both for the scattering relation and for $L$. We also show that instead of travel times, one can take any defining function of pairs of boundary points which can be connected by a null bicharacteristic, and its linearization leads to this light ray transform. Linearizing such a defining function is more convenient than linearizing the scattering relation, see also \cite{S-Lorentzian-scattering}. 

We present an application where this phase space analysis can be used to solve a specific rigidity problem. We consider Finsler manifolds and prove rigidity results using the Legendre transform, which allows us to reduce Finsler metrics to Hamiltonians. Going back, from Hamiltonian to Finsler, requires $H_{\xi\xi}>0$, which we do not assume in our Hamiltonian analysis, but which we need here. In applications to anisotropic elasticity, for example, this condition holds for the Hamiltonian corresponding to waves traveling at fastest speed, but not in general. 
We derive our main results using the Hamiltonian formalism only, and a bit of symplectic geometry. Works treating rigidity problems for Finsler manifolds that we are aware of are \cite{HoopIL2021, HoopIL2023}, where the formulation is different: one assumes either internal sources or internal reflections. 
The group preserving the data in \cite{HoopIL2021, HoopIL2023}  consists of isometries on the base fixing $\bo$ pointwise. In contrast, here we have a much larger group of canonical transformations in the phase space that leave the data unchanged, and as we discuss in Section \ref{sec_Finsler}, there exist such transformations which do not arise from diffeomorphisms of the base manifold. 
This also shows that the travel time problem in elasticity cannot be solved up to base diffeomorphisms fixing the boundary by viewing it as a problem in Finsler geometry.  
A boundary rigidity problem for a certain subclass of Finsler metrics was also studied in \cite{MonkonenRandersMetrics}, where it was also remarked that there are non-isometric Finsler metrics with the same boundary distance function.

We formulate and prove most of our results in a semi-global way, with local data and local recovery. With appropriate assumptions, one can formulate them in a global way as well. 
This would require more careful analysis of the smoothness of the corresponding canonical transformations and X-ray transforms  near glancing rays.   

This paper is organized as follows. We study Hamiltonian systems at positive energy levels in section~\ref{sec_2}. We introduce the scattering relation and the travel times there, the latter parameterized by initial points and directions. We prove a semi-global rigidity result in Theorem~\ref{thm_H}. In section~\ref{sec_3}, we parameterize the travel times $T(x,y)$ by endpoints under a non-conjugacy assumption, and linearize them with respect to the Hamiltonian. The result is the Hamiltonian X-ray transform, whose kernel we describe. Although we do not need the linearization to solve the nonlinear problem, which we did in section~\ref{sec_2}, we do this to connect it to the metric case, and hope this to be useful for the study of Hamiltonians of special classes. We also show in this section that the travel times $T(x,y)$ at the boundary recover the scattering relation uniquely, serving in fact as a generating function.  We consider zero energy levels in section~\ref{sec_zero},  solving the nonlinear rigidity problem and the linearized ``light ray'' X-ray problem.  The Finsler rigidity result is proven in section~\ref{sec_Finsler}. 

\smallskip

\noindent\textit{Acknowledgments.} N.E. would like to thank A.P. Contini for helpful discussions.

\section{Positive energy levels. The scattering relation and the travel times} \label{sec_2}

\subsection{Preliminaries} 
Let $M$ be a smooth manifold of dimension $n\ge2$, with or without boundary. Let $H(x,\xi)\in C^\infty(T^*M\set)$ be a positive and  homogeneous\footnote{Throughout the paper, ``homogeneous'' means positively homogeneous, always with respect to the fiber variable. } of degree two Hamiltonian, i.e.,  $H\circ \mathcal{M}_\lambda=\lambda^{2}H$ for $\lambda>0$, where 
\begin{equation}
\label{eq:dilation}
	\mathcal{M}_\lambda:T^*M\set\to T^*M\set,\quad (x,\xi)\mapsto (x,\lambda \xi)
\end{equation}
is the dilation operator. 
We can think of the energy level $H=1/2$ as the reference level, which defines the travel times of interest. An example is $H=\frac12 g^{ij}(x) \xi_i \xi_j$ with $g=\{g^{ij}\}$ being  the co-metric corresponding to a Riemannian metric, written in local coordinates. We call this the ``metric case.'' 
We fix the symplectic two form on $T^*M$ to be $\sigma= \d\xi\wedge \d x$ with corresponding matrix $J=(0,-\Id;\Id,0)=-J^{\mathsf{T}}$ in block form, and let $X_H=-J(\d H)^{\mathsf{T}}=H_\xi\partial_x-H_x\partial_\xi$ be the related Hamiltonian generator of the Hamiltonian flow  $\Phi^t$. %
(Here and throughout, total differentials are  row vectors.) 
 We assume that $H$ is non-degenerate at energy level $1/2$ (and hence, at any positive one), i.e., $\d H\not=0$ on $H=1/2$. We denote by $\alpha = \xi\d x$ the tautological 1-form on $T^*M\set$, for which $\sigma=\d \alpha$.

Euler's identity $\xi\cdot H_\xi=2H$ indicates that at positive energy levels, $X_H$ is never radial (proportional to $\xi\partial_\xi$,  and in fact, $\xi\partial_\xi$ is always transversal to the positive energy level hypersurfaces. 
We say that the Hamiltonian curve $(x(t),\xi(t))$ (or its $x$-projection) is issued from $(x_0,\xi^0)$ if the latter is an initial condition at $t=0$. The actual direction at $t=0$ is $v \coloneqq  \dot x = H_\xi$, of course. Note that the map $\xi\mapsto v$ is not necessarily a local diffeomorphism; for that we would need $\det H_{\xi\xi}\not=0$. 
Define 
\begin{equation}   \label{norm}
|\xi|_x= (2H(x,\xi))^\frac1{2}, \quad S^*M = \{(x,\xi)\in T^*M\set\;|\; |\xi|_x=1\}. 
\end{equation} 
Due to the non-degeneracy assumption $\d H\neq 0$ whenever $H=1/2$, $S^{*}M$ is a smooth submanifold of $T^*M$. Moreover, since $H_\xi\neq 0$ whenever $H=1/2$ by Euler's identity, each fiber $S_x^*M$ for fixed $x$ is a submanifold of $T_x^{*}M$.
Note that $|\xi|_x$ is not necessarily a norm, and it may not even satisfy $|-\xi|_x= |\xi|_x$.  On the other hand, $|\lambda \xi|_x =\lambda |\xi|_x $ for every $\lambda>0$, and $|\xi_x|>0$ when $\xi\not=0$. In the metric case, $|\xi|_x$ is the length of the covector $\xi$.  We call  bicharacteristics at energy level $1/2$ \textit{unit}; then $|\xi|_x=1$.  
They are related with bicharacteristics on other positive energy levels via the following lemma.

\begin{lemma}  \label{lemma_hom}
If $\gamma(t) = (x(t),\xi(t))$ is a Hamiltonian curve at energy level $E$, then for every $\lambda \not=0$, $\tilde\gamma(t)=\mathcal{M}_\lambda \circ \gamma(\lambda t)=(x(\lambda t), \lambda\xi(\lambda t))$ is a Hamiltonian curve at energy level $\lambda^2E$.
In terms of the flow  map, $\Phi^t\circ \mathcal{M}_\lambda= \mathcal{M}_\lambda\circ \Phi^{\lambda t}$.
\end{lemma}

\begin{proof}
It is straightforward to check that $  \d \mathcal{M}_{\lambda}J=\lambda J(\d \mathcal{M}_{\lambda^{-1}} )^{\mathsf{T}} $. By homogeneity of $H$,
\begin{align*}
	\frac \d {\d t}\tilde \gamma &=\lambda \d \mathcal{M}_\lambda X_H\big|_{\tilde \gamma(t)}=-\lambda \d \mathcal{M}_\lambda J(\d H)^{\mathsf{T}}\big|_{\tilde \gamma(t)}=-\lambda^{2} J(\d \mathcal{M}_{\lambda^{-1}} )^{\mathsf{T}}(\d H)^{\mathsf{T}}\big|_{\tilde \gamma(t)}\\
	&=- J(\d H)^{\mathsf{T}}\big|_{\tilde \gamma(t)}
	=X_H\big|_{\tilde \gamma(t)},
\end{align*}
which shows the claim.
\end{proof}

\noindent For example, when $H=|\xi|^2/2$ on $T^*\R^n$, the Hamiltonian curves are of the form $\gamma(t)=(x_0+t\xi^0,  \xi^0)$. Rescaling as above, we get $\tilde \gamma(t)=(x_0+(\lambda t) \xi^0,\lambda \xi^0) = (x_0+t(\lambda \xi^0),\lambda\xi^0)$ which is a Hamiltonian curve with the initial condition $(x_0,\lambda\xi^0)$, at  energy level $\lambda^2|\xi^0|^2/2$.

\subsection{The scattering relation and the travel times}\label{sub:the_scattering_relation_and_the_travel_times}
Throughout  and unless otherwise stated, let $M$ be a manifold with boundary.
We say that $(x,\xi)\in T^*M\set\big|_{\partial M}$ is incoming (resp. outgoing) if  the initial velocity vector  $H_\xi(x,\xi) = \d \pi X_H$ of the corresponding bicharacteristic (where $\pi:T^{*}M\to M$ is the projection) is inward (resp. outward) pointing.
Fix two disjoint open subsets $U$ and $V$ of $\partial M$ and let $(x_0,\xi^0)\in T^*M\set$ with $x_0\in U$ be an initial covector giving rise to a unit bicharacteristic $\gamma_0$, and assume that $\gamma_0(s)$ hits $V$ (more precisely, $T^*M\set\big|_V$) transversally for time $s=s(x_0,\xi^0)\not=0$ (positive or negative) at some $(y_0,\eta^0)$ for the first time, see also Figure~\ref{fig2a}. 
\begin{figure}[h!] %
  \centering
  \includegraphics[scale=1,page=6]{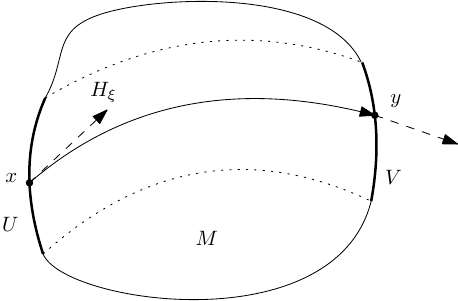}
\caption{%
The scattering relation $S$. The covectors $\xi$, $\xi'$, $\eta$ and $\eta'$ are not plotted. The domain between the dotted lines represents $\Gamma$, see \eqref{eq:Gamma} (projected to the base). 
}
\label{fig2a}
\end{figure}
The transversality is preserved for each bicharacteristic $\gamma_{x,\xi}$ issued from  $(x,\xi)$ in a small conic neighborhood of $(x_0,\xi^0)$, with $x\in U$, and an endpoint $(y,\eta)$ at $s(x,\xi)$ with $y\in V$. As a first step, we can define the local scattering relation
\begin{equation}   \label{a2}
	\hat S : (x,\xi)\mapsto (y,\eta).
\end{equation}
We also set
\begin{equation}   \label{a3}
	\text{travel time} = \hat\ell(x,\xi ).
\end{equation}
If $(x,\xi)$ is incoming (resp. outgoing),  so that  $(y,\eta)$ is outgoing (resp. incoming), then $\hat\ell(x,\xi)>0$ (resp. $\hat\ell(x,\xi)<0$).
Note that since $H$ is not necessarily symmetric with respect to the antipodal map in the fibers, its flow needs not be reversible. Therefore, if we know that $\hat{S}(x,\xi)=(y,\eta)\in T^*M\big|_V$ for some positive time $s(x,\xi)$, there is no reason why a bicharacteristic  starting at $(x,-\xi)$ would intersect $T^*M\big|_V$ for any time, positive or negative.

In view of the intended applications to PDEs, the scattering relation and the travel times are extracted from wave front sets on $T^*\bo$. For this reason, we parameterize those quantities with such covectors. 
The pullback by the canonical inclusion $\iota:\partial M\hookrightarrow M$ induces a  map $\iota^{*}: T^*M\to T^*\partial M$, which corresponds to restricting an element of $T^*M\big|_{\partial M}$ to $T\partial M$.
We pull $(x,\xi)\in T^*M\big|_U$ and $(y,\eta)\in T^*M\big|_V$ back to $U$ and $V$, respectively, denoting the pullbacks with primes, that is,
\begin{equation}\label{eq:pullbacks}
 	(x,\xi')=\iota^{*}(x,\xi),\quad (y,\eta')=\iota^{*}(y,\eta).
 \end{equation} 
In terms of boundary coordinates $(x',x^n)$ for the base  with $x^n=0$ on $U$ and corresponding linear coordinates $(\xi,\xi_n)$ for $T^*M$, $\iota^{*}(x,\xi)=(x',\xi')$. 

Given $\xi'$ at $x\in U$, it is not true for a general Hamiltonian that there is a unique unit $\xi$, say incoming or outgoing, with that pullback, smoothly depending on $(x,\xi')$. The ``slowness surface'' $|\xi|_x=1$ is not necessarily strictly convex or even convex, see Figure~\ref{fig1}. To guarantee existence of unique $\xi$ with a given $(x,\xi')$ \textit{locally}, we impose on $H_\xi=\d \pi X_H$ the following condition:
\begin{equation}   \label{A1}
\langle \d \rho,H_\xi\rangle\not=0 %
\end{equation}
at $(x_0,\xi^0)$, where $\rho$ is a boundary defining function\footnote{A \textit{boundary defining function} $\rho$ on a manifold with boundary $M$ by definition satisfies $\rho>0$ on $M^{\circ}$, $\rho=0$ on $\partial M$, and $\d \rho\neq 0$ on $\partial M$.}.
 Passing to boundary local coordinates in which $U=\{x^n=0\}$, condition~\r{A1} takes the form $H_{\xi_n} (x_0,\xi^0) \not=0$. By the Hamiltonian system, this can be written as $\dot x^n\not=0$. Then the bicharacteristic $\gamma_{x_0,\xi^0}(s)$ is incoming/outgoing if and only if  $\pm \langle \d \rho,H_\xi\rangle>0$ but never tangent to $U$. In Figure~\ref{fig1}, for example,  $\langle \d x^{n} ,H_\xi\rangle<0$, thus $\xi$ is outgoing. 
Under condition~\r{A1}, knowing  $\xi'=(\xi_1,\dots,\xi_{n-1})$ at some $x\in U$,  uniquely recovers $\xi_n$ locally given   $|\xi|_x= (2H(x,\xi))^{1/2}=1$ by the Implicit Function Theorem; and then it recovers $\xi$. Denote this solution by 
\begin{equation}   \label{zeta}
\xi = \zeta(x',\xi') = (\xi',\zeta_n(x',\xi')). %
\end{equation}
This defines a local inverse $\zeta=\zeta_{(x_{0},\xi^{0})}$ of the pullbacks in \eqref{eq:pullbacks}.
In the metric case, $\xi_n=(1-|\xi'|_g^{2})^{1/2}$, written in boundary normal coordinates. 
We also note that in the Hamiltonian case, an inequality of the form $|\xi'|<1$  does not imply that $|\xi|=1$ is solvable for a given $\xi'$. In fact, we work locally near some $(x_0,\xi^0)$ and we do not even assume that $H(x_0,\xi^0{}')$ is defined. 

\begin{figure}[h!] %
  \centering
  \includegraphics[scale=1,page=1]{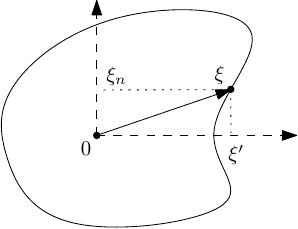}
\caption{%
	The slowness surface $|\xi|_x=1$ (in 2D) in dual coordinates associated with boundary coordinates $(x',x^n)$ at $U$, where locally $x^n$ is a boundary defining function. Condition \r{A1} means that the vertical line dropped from the vertex of $\xi$ is transversal to the surface there. It does not exclude existence of other $\xi$'s pointing into $M$ with the same restriction $\xi'$. For the $\xi $ pictured above, $\d x_n H_\xi =H_{\xi_{n}}<0$, indicating that $\xi$ is outgoing.
}

\label{fig1}
\end{figure}

We assumed above that the unit bicharacteristic $\gamma_0(s)$ issued from $(x_0,\xi^0)$, with \r{A1} satisfied, hits $T^*M\set\big|_V$ at $s=s_0\not=0$, transversely. The latter means that $x(s_0)\in V$, and $\dot x(s_0)$ is not tangent to $V$ (and in particular, does not vanish there), where, temporarily, $x(s)$ is the space component of $\gamma_0$.  Set $(y_0,\eta^0)=\gamma(s_0)$. Then we have \r{A1} at $(y_0,\eta^0)$ as well. Clearly, the transversality condition \r{A1} at $x_0$ and $y_0$ is stable under small perturbations.  

\begin{definition} \label{def_Ls}
Let $(x_0,\xi^0)\in S^*M$ with $x_0\in U$, satisfy  \r{A1}. For $(x,\xi')\in T^*U$, close enough to $(x_0,\xi^0{}')$, let $(x,\xi)$ be the locally unique solution to $\iota^{*}(x,\xi)=(x,\xi')$ and $|\xi|_x=1$ in a neighborhood of $(x_{0},\xi^{0})$, and let $\gamma_{x,\xi}(s)$ be the unit bicharacteristic issued from $(x,\xi)$. Assume that $\gamma_{x,\xi}(s)$ hits $V$ transversely for $s=\ell(x,\xi')\not=0$ at $y\in V$, with a restriction $\eta'$ of its fiber there to $TV$. Then we call
\begin{equation}\label{eq:scattering_relation}
  S:(x,\xi') \to (y,\eta'),  
\end{equation}
the \textbf{scattering relation}, and $\ell(x,\xi')$ the associated \textbf{travel time} (positive or negative).
\end{definition}
\noindent Note that with our notation,
\begin{equation}
	S=\iota^{* }\circ \hat S \circ \zeta.
\end{equation}
Whenever we need to emphasize the dependence of $S$ or $\hat S$ on $H$, we will add it as a superscript.  
The pair $(S,\ell)$ is called \textit{lens data}.  We follow the accepted terminology, see \cite{Croke_scatteringrigidity}, for example. Perhaps the terms ``scattering'' and ``lens'' should be swapped since time of propagation plays no role in geometric optics for optical lenses, and ``scattering'' can be applied to spacetime.

To make the localization more specific, let $\mathcal{U}\subset T^*U$ be a neighborhood of $(x_0,\xi^0{}')$ for which the local solvability assumption  with respect to $(x_0,\xi^0)$  and the transversality assumption at $V$ in Definition~\ref{def_Ls} hold. We denote the set of the restrictions to $TV$  of the corresponding endpoints of the bicharacteristics over $V$ by $\mathcal{V}\subset T^*V$.  The set $\mathcal{V}$ is open as well since $S:\mathcal{U}\to \mathcal{V}$ is a diffeomorphism (symplectic, in fact), having as inverse the scattering relation on $\mathcal{V}$ relative to $\hat{S}(x_{0},\xi^{0})$. %

So far,  we related $S$ and $\ell$ to an energy level $1/2$. We can extend $S$ to a $\lambda$-dependent family $S_\lambda(x,\xi')$ with $\lambda >0$, requiring it to be homogeneous with respect to $\xi'$ of order one, by setting $S_\lambda(x,\xi') = \mathcal M_\lambda (S(x,\xi'/\lambda))$.
We also set $\ell_\lambda(x,\xi') = \lambda^{-1}\ell(x,\xi'/\lambda )$.  Then $S_\lambda$ and $\ell_\lambda$ correspond to Hamiltonian curves at energy level $\lambda^2/2$, and are uniquely determined by $S$, $\ell$. Next, we set  
\begin{equation}   \label{U_l}
\mathcal{U}_\lambda = \{(x,\xi')\in T^*U|\; (x,\xi'/\lambda)\in \mathcal{U} \}.
\end{equation}
We have a  map $\zeta_{\lambda}:\mathcal{U_\lambda}\to T^*M\set\big|_U$, which extends \eqref{zeta} by homogeneity (of order 1).

In the metric case, 
$\mathcal{U}$ can be taken to be the open unit ball  bundle in $T^*U$, denoted $B^*U$, and similarly for $\mathcal{V}$.  In boundary normal coordinates, $2H(x,\xi) = %
|\xi'|_g^2 +\xi_n^2$, and $2H=1$ can be solved smoothly for $\xi_n$ when $\xi_n\not=0$, which in this case is equivalent to \r{A1}. The two solutions correspond to incoming/outgoing rays. For $\lambda>0$, we solve $2H=\lambda^2$ to define $S_\lambda$. %

Let $(x_0,\xi^0)\in T^*M\set\big|_U$ be as in Definition \ref{def_Ls}; for simplicity we assume throughout that it is incoming (the outgoing case can be addressed similarly).
Let $\Gamma_{(x_0,\xi^0)}\subset T^*M\set$ be the conic set 
\begin{equation}\label{eq:Gamma}
	\begin{split}
		&\Gamma_{(x_0,\xi^0)} \coloneqq\bigcup _{\lambda>0}  \big\{(y,\lambda\eta)|\; \text{$(y,\eta)$ lies on the unit bicharacteristic issued from}\\
		&\qquad  \text{ some $(x,\xi)$ in a small neighborhood of $(x_0,\xi^0)$ in $S^*M$ with $(x,\xi')\in\mathcal{U}$} \big\},
	\end{split}
\end{equation}
that is, $\xi$ is the locally unique solution of $H(x,\xi)=1/2$ with restriction to $T_xU$ given by $\xi'$. 
Whenever $ (x_0,\xi^0)$ is fixed throughout a certain discussion we will suppress it in the notation and write it as $\Gamma$.
Note that $S_\lambda$ is defined by Hamiltonian curves passing through $\Gamma$ only, and it could only possibly recover $H$ there. 
Throughout, for $(y,\eta)\in \Gamma $ we write 
\begin{equation}\label{eq:ell_pm}
	\ell_- (y,\eta)=\inf \{t\geq 0: \pi\circ \Phi^{-t}(y,\eta)\in U\}, \quad \ell_+ (y,\eta)=\inf \{t\geq 0: \pi\circ \Phi^{t}(y,\eta)\in V\}.
\end{equation} 
By Lemma \ref{lemma_hom}, $\ell_\pm$ are homogeneous of order $-1$.

\subsection{Rigidity results} 
It is well known, see, e.g., \cite[Ch.~45A]{Arnold_Mech}, that given a canonical transformation $(\tilde x,\tilde \xi)=\kappa(x,\xi)$,   the Hamiltonian $\tilde H = H \circ\kappa^{-1}$ 
defines a new Hamiltonian system related to $\tilde H$ with a solution $(\tilde x(s),\tilde \xi(s)) = \kappa( x(s),\xi(s) )$ for any solution $( x(s),\xi(s) )$ of the original one. This can be written as $\kappa\circ\Phi^s= \tilde\Phi^s\circ\kappa$, where $\tilde \Phi^s$ is associated with $\tilde H$, see also \r{Phi} below. Equivalently,  $X_{\tilde H}=\d \kappa X_H $. Here and below, a canonical transformation stands for a symplectic diffeomorphism from its domain (always excluding the zero section) onto its image. 
We are especially interested in homogeneous diffeomorphisms on conical domains, i.e., these that intertwine the dilation operator in the domain and image respectively.
Homogeneous canonical transformations preserve the homogeneity of Hamiltonians and correspond in a suitable sense to contactomorphisms between the corresponding cosphere bundles, see \cite[Appendix 4]{Arnold_Mech}. 

One can show that a diffeomorphism $\kappa$ between open subsets of $T^*M\set$ conjugates the flows of $H$ and $H\circ \kappa^{-1}$ for all $H$ if and only if it is a canonical transformation. Indeed, with the symplectic form written as $J=(0,-\Id;\Id,0)$ in block form in terms of Darboux coordinates as before, the symplectic condition for $\kappa$ can be written as $(\d\kappa)^\mathsf{T}J(\d\kappa)=J$, equivalently  $(\d\kappa)J(\d\kappa)^\mathsf{T}=J$. Recall, $X_H=-J\d H^{\mathsf{T}}$.
Writing  $\tilde H = H \circ\kappa^{-1}$, we have $X_{\tilde H} = -J( \d\kappa^{-1})^\mathsf{T}\d H^\mathsf{T}$. If $\kappa$ is canonical, then $X_{\tilde H} =-\d\kappa J \d H = \d\kappa X_H $, confirming the conjugacy of the flows. On the other hand, if the flows are conjugated for every Hamiltonian $H$, then for all $H$ we have $J( \d\kappa^{-1})^\mathsf{T}\d H^{\mathsf{T}}= \d\kappa J\d H^{\mathsf{T}}$, which implies the symplectic condition for $\kappa$. 

Let $\tilde M$ be a manifold with boundary, having the same boundary as $M$.
Having two Hamiltonians $H$ and $\tilde H$ on $T^*M\set$ and $T^*\tilde M\set$ respectively, we denote the objects corresponding to $\tilde H$ with tildes over them. 
Assume that $(x_0,\tilde \xi^0)\in T^*\tilde M\set$ with $x_0\in U$ is such that \r{A1} holds there for $\tilde H$; then $(x_0,\tilde \xi^0)$ gives rise to a conic set $\tilde \Gamma\subset T^*\tilde M\set $ as in \eqref{eq:Gamma}. Let $\Gamma_U$ be the collection of all $(x,\xi)$ with $x\in U$, $(x,\xi)\in\Gamma$; and define $\Gamma_{V}$, $\tilde \Gamma_{ U}$, $\tilde \Gamma_{V}$ similarly. 
Then for every $\kappa$ homogeneous and canonical,  with $H=\kappa^{*}\tilde H$, if
\begin{equation}   \label{kappa_b}
\kappa(\Gamma_U) =  \tilde \Gamma_U, \quad \kappa(\Gamma_V) =  \tilde \Gamma_V, \quad \text{and $\iota^*\circ \kappa= \iota^*$ on $\Gamma_U\cup\Gamma_V$},
\end{equation}
then $S=\tilde S$ on $\mathcal{U}$. A map $\kappa $ as above is  the analog of a diffeomorphism on $M$ fixing $\p M$ pointwise in the metric setting, and the pullback by the differential of such a diffeomorphism is an example of a canonical transformation $\kappa$ as above.  
Note that $\pi\circ \kappa(x,\xi)=x$ for all $(x,\xi)\in \Gamma_U$ by the first and third property in \eqref{kappa_b}, and similarly for $V$. There is a rich non-trivial set of such transformations, see, e.g.,  Remarks~\ref{rmk:convex} and \ref{rem_can}.

The converse is also true, which is our semi-global rigidity theorem. 

\begin{theorem}[semi-global rigidity for Hamiltonian systems at a positive energy level] \label{thm_H} 
With notations as before, assume that $ S^{H}=S^{\tilde H}$, $ \ell=\tilde\ell$ on $\mathcal{U}$ (relative to incoming covectors $(x_{0},\xi^{0})$, $(x_{0},\tilde\xi^{0}) $). Then there exists a homogeneous canonical transformation $\kappa: \Gamma\to \tilde \Gamma$ so that 
\r{kappa_b} holds, 
and one has $ \kappa^*\tilde H=  H$ on $ \Gamma$.
\end{theorem}

\begin{remark}
Locally, every two non-degenerate Hamiltonians are equivalent by a pull-back %
of a canonical transformation, by the Darboux Theorem, see \cite[Theorem~21.1.6]{Hormander3}. Indeed, one can start with $p_1=H$ in a neighborhood of some $(x,\xi)$ and complete it to a symplectic coordinate system $(q,p)$
in which $X_H= \p/\p q^1$. When $H$ is homogeneous, and $X_H$ is never radial (as in our case),  
then one can choose the transformation to be homogeneous as well. Indeed, one can apply   the homogeneous Darboux theorem \cite[Theorem~21.1.9]{Hormander3},  \cite[Prop. 6.1.3]{Hormander_actaII}, to $H_1\coloneq (2H)^{1/2}$, which is never radial by the Euler identity again, to conclude that $H_1$ is  equivalent to $G_1\coloneq \xi_n$ in a conic neighborhood of any $(x_0,\xi^0)\in T^* M\set$; then $H$ is equivalent to $G \coloneq \xi_n^2/2$ with $X_{G}=\xi_n\p_{x^n}$ (with $n$ being the dimension, no summation), or to $|\xi|^2/2$ if we wish. The theorem says that under its assumptions, there exists such a semi-global transformation satisfying certain boundary conditions. Also, it says that we cannot extract anything more from the lens data.
\end{remark}

\begin{proof}[Proof of Theorem~\ref{thm_H}]  
The proof follows an argument, known at least in the metric case, which establishes conjugacy of the flows  under a map in the phase space, see, e.g., \cite{Vargo_09,BohrMonardPaternain}. 
To illustrate the idea, we start first with an easier setup, namely under the stronger assumption that %
$H=\tilde H$ to infinite order at $\partial TM\set $, as well as $\Gamma_U=\tilde \Gamma_U$, and $\Gamma_V=\tilde \Gamma_V$ (so that $\zeta=\tilde\zeta$) and treat the general case afterwards. 
Under these assumptions, 
we can embed $M$ into a larger open manifold $N$ of the same dimension and extend $H$ and $\tilde H$ smoothly to $N\setminus M$  in such  a way that $\tilde H=H$ there. 
Set
\begin{equation}\label{eq:kappa_map}
\kappa(y,\eta)\coloneqq  \tilde \Phi^{\ell_-(y,\eta)}\circ \Phi^{-\ell_-(y,\eta)}(y,\eta),  	
  \end{equation}  %
see Figure~\ref{fig2}, which is  homogeneous by Lemma \ref{lemma_hom} and takes values in $\tilde \Gamma \subset T^{*}\tilde M$ by the equality of the travel times. 
We check that it is invertible. Let $(\tilde y,\tilde \eta)=\kappa(y,\eta)$, which  
lies in the orbit of $(x,\xi)\coloneqq\Phi^{-\ell_-(y,\eta)}(y,\eta)\in \Gamma_U=\tilde \Gamma_U$.
Substituting $(\tilde y,\tilde \eta)=\tilde \Phi^{\ell_-(y,\eta)}(x,\xi)$ into the statement $\tilde \Phi^{-\tilde \ell_-(\tilde y,\tilde \eta)}(\tilde y,\tilde \eta)\in \tilde \Gamma_U$ 
we see that
$\tilde \Phi^{\ell_-(y,\eta)-\tilde \ell_-(\tilde y,\tilde \eta)}(x,\xi)\in \tilde \Gamma_U$.
We have $\ell_-(y,\eta)\geq \tilde \ell_-(\tilde y,\tilde\eta)$ by definition of $\tilde \ell_-$; if $\ell_-(y,\eta)>\tilde\ell_-(\tilde y,\tilde\eta)$, the bicharacteristic $\gamma_{(x,\xi)}$ intersects $\partial T^*\tilde M$ at an incoming covector for positive time. This contradicts the equality of travel times, according to which  $\tilde \Phi^{t}(x,\xi)$ does not leave $T^*M$ for $t\in [0,\ell_-(y,\eta)]\subset [0,\hat\ell(x,\xi)]$. 
We find that $\ell_-(y,\eta)=\tilde \ell_-(\tilde y,\tilde \eta)$, thus $\kappa$ is inverted by $\kappa^{-1}(\tilde y,\tilde \eta)=  \Phi^{\tilde \ell_-(\tilde y,\tilde \eta)}\circ \tilde \Phi^{-\tilde \ell_-(\tilde y,\tilde \eta)}(\tilde y,\tilde\eta).$
Further, it is a canonical transformation: indeed, given $(y,\eta)\in \Gamma$, we can find a locally constant  $t_0>\ell_-(y,\eta)$ such that 
$\Phi^{-t_0}(y,\eta)\in N\setminus M$ and upon noticing that  $\kappa(y,\eta)= \tilde \Phi^{t_0}\circ \Phi^{-t_0}(y,\eta)$ since $H=\tilde H$ on $N\setminus M$,  the symplectic property holds because $\tilde \Phi^{t_0}$ and $\Phi^{-t_0}$ are both symplectic.
By construction and the equality $\ell=\tilde \ell$, $\kappa=\Id$ on $\Gamma_U$ and  $\kappa(\Gamma_V)=\tilde \Gamma_V$. 
Moreover, $\iota^*\circ \kappa=\iota^{*}$ trivially on $\Gamma_U$, and on $\Gamma_V$ we have, with notation as above,
\[\begin{aligned}
 	\iota^{*} \kappa(y,\eta)=\iota^{*}\tilde \Phi^{\ell_-(y,\eta)}(x,\xi)=\iota^{*}\tilde \Phi^{\hat\ell(x,\xi)}(x,\xi)=\iota^{*}{\hat {S}}^{^{\tilde H}}(x,\xi)= S^{\tilde H}(x,\xi')\\ =S^{H}(x,\xi')=\iota^{*}\hat S^{H}(x,\xi)=\iota^{*}\Phi^{\hat\ell(x,\xi)}(x,\xi)=\iota^{*}(y,\eta).
 \end{aligned}
 \]
Since a  Hamiltonian stays constant along its bicharacteristic flow and $H=\tilde H$ on $T^*M\set\big|_U$, 
$\kappa^*\tilde H=H$ on $\Gamma$. 
We also mention that $\kappa$ conjugates the flows:  with $t_0$ as before, 
\begin{equation}
\begin{aligned}\label{Phi}
\tilde \Phi^s\circ\kappa ( y , \eta) & = \tilde \Phi^{t_0+s}\circ \Phi^{-t_0}(y,\eta) = \tilde \Phi^{t_0+s}\circ \Phi^{-t_0-s} \circ \Phi^{s}  (y,\eta)\\
& = \tilde \Phi^{t_0}\circ \Phi^{-t_0} \circ \Phi^{s}  (y,\eta)= \kappa\circ  \Phi^{s}  (y,\eta).
\end{aligned}
\end{equation}
\begin{figure}[h!] %
  \centering
  \includegraphics[scale=1,page=3]{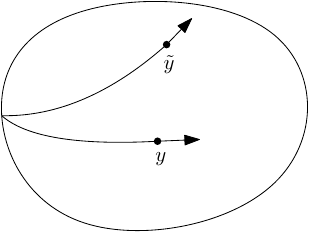}
\caption{%
The canonical relation $\kappa$. The codirections are not plotted.  
}

\label{fig2}
\end{figure}

We now treat the general case. Here we need not have $\Gamma_U=\tilde \Gamma_U$, hence we will replace \eqref{eq:kappa_map} with  
\begin{equation}\label{eq:kappa_2}
	\kappa= \tilde \Phi^{\ell_-(y,\eta)}\circ \tilde \zeta_{\lambda}\circ \iota^{*}\circ \Phi^{-\ell_-(y,\eta)},\quad \lambda = \sqrt{2H(y,\eta)};
\end{equation} the arguments below aim to show its symplectic property.
Choose local coordinates $x'$ on $U$, and let $\xi'$ be the dual ones, in $T^*U$. Define the map $\Psi: ( x',s,\xi',E)\mapsto (y,\eta)$
as follows 
\begin{equation}  \label{Psi}
(y,\eta) = \Phi^{s}((x',0), \xi), 
\qquad  \xi=(\xi',h(x',\xi',E)),\quad E>0,\quad (x',\xi')\in\mathcal{U}_{\sqrt{2E}}, \quad s\in \big[0,\ell_{\sqrt{2E}}(x',\xi')\big].
\end{equation} 
The function $h(x',\xi',E):=(\zeta_{\sqrt{2E}})_{n}(x',\xi')$ 
satisfies
$
H(x',0,\xi',h)=E$
and we have the homogeneity property $h(x',\xi',E)={(2E)}^{1/2}h(x',{(2E)}^{-1/2}\xi',1/2) $.

We claim that $\Psi$ is a canonical transformation on its domain. 
We check the symplectic property at $s=0$ first by   computing the Jacobian there. 
By the implicit relation, we have
\begin{equation}
  \partial _{x'} h=-\frac{\partial_{x'}H}{\partial_{\xi_n }H},\quad \partial _{\xi'} h=-\frac{\partial_{\xi'}H}{\partial_{\xi_n}H},\quad \partial _{E} h = \frac{1}{\partial_{\xi_n}H}.
\end{equation}
We differentiate in $s$ at $s=0$, writing $y=(y',y^n)$ and $\eta=(\eta',\eta_n)$:
\begin{equation}
  \partial_{s} y= \d y X_H = \partial_\xi H,\quad \partial_{s} \eta= \d \eta X_H=- \partial_x H.
\end{equation}
For all other derivatives, when $s=0$, we have $(y',y^n,\eta',\eta_n)=(x',0,\xi',h(x',\xi',E ))$.
Combining the above, the Jacobian for $s=0$ takes the form 
\begin{equation}\label{eq:matrix}
  D_{x',s,\xi',E}(y',y^n,\xi',\eta_n)=
\begin{pmatrix}
	\Id &  \d _{\xi'} H^{\mathsf{T}} & 0 & 0\\
	0 & \partial_{\xi_n} H & 0 & 0\\
	0 & -\d_{x'}H^{\mathsf{T}}  & \Id & 0\\
	-\frac{\d_{x'}H}{\partial_{\xi_n }H} & -\partial_{x^{n}}H & -\frac{\d_{\xi'}H}{\partial_{\xi_n}H} & \frac{1}{\partial_{\xi_n}H}
\end{pmatrix}
  =\begin{pmatrix}
    \Id & C & 0 & 0 &\\
    0 & a & 0 & 0\\
    0 & D & \Id & 0\\ 
    D^{\mathsf{T}}/a & b & -C^{\mathsf{T}}/a & 1/a
  \end{pmatrix},
\end{equation}
where the lowercase letters denote scalars and $C,$ $D$ have dimensions $ (n-1)\times 1$. 
It is now a straightforward computation to check that if $J=(0, -\Id; \Id ,0) $ in block form as before and $A$ is any matrix of the form \eqref{eq:matrix}, one has $A^{\mathsf{T}} JA=J$.
So the map \eqref{Psi} satisfies the symplectic property at $s=0$.
We can now propagate this property away from $s=0$. 
For any fixed $s_0$, 
\begin{equation}\label{eq:symplectic_psi}
	\d \Psi\big|_{s=s_0}=\d (\Phi^{s}(x',0,\xi',h))\big|_{s=s_0}=\d (\Phi^{s_0}\circ \Psi(x',s-s_0,\xi',E))\big|_{s=s_0}=\d \Phi^{s_0}\circ\d  \Psi\big|_{s=0},
\end{equation}
which is symplectic as a composition of symplectic maps. 

We have constructed a canonical transformation $\Psi: (x',s,\xi', E) \mapsto (  y,   \eta)$ which is bijective from  the domain indicated in \r{Psi} to $\Gamma$ by the definition of $\Gamma$. 
It has the homogeneity property
\begin{equation}
	\Psi (x',s,\xi', E)=(y,\eta)\implies \Psi(  x',\lambda^{-1}s,\lambda\xi',\lambda^{2} E)=(y,\lambda \eta),\quad \lambda>0.
\end{equation}
Given $\tilde H$, we construct the same transform related to it. 
Then \eqref{eq:kappa_2} is given by the composition
\[
\kappa: (y,\eta) \overset {\Psi^{-1}}{\longmapsto} (x',s,\xi', E) \overset {\tilde \Psi}{\longmapsto}  (\tilde y, \tilde \eta)
\]
mapping $\Gamma$ to $\tilde\Gamma$, 
which is canonical, and also homogeneous of order one. 
 We still have $\kappa^{*}\tilde H=H$ since  $H(y,\eta) = E=\tilde H(\tilde y,\tilde\eta)=\tilde H(\kappa(y,\eta))$. The rest of the properties are easily verified. 
\end{proof}

\begin{remark}\label{rem_hom}
	Alternatively, we could have considered a change of coordinates $(x',t,\xi',\lambda)\mapsto (y,\eta)\in \Gamma$, given by 
\begin{equation}\label{eq:psi_homog}
	(y,\eta) = \Phi^{t/\lambda}((x',0), \lambda \zeta(x',\xi'/\lambda)),
\end{equation} where $\lambda: =|\eta|_y= (2H(y,\eta))^{1/2}$.
We could then have proceeded in the same way to find that this coordinate change corresponds to a \emph{homogeneous} canonical transformation (unlike our construction in the proof above). 
The current approach however is easier to adapt to the zero energy case in Section \ref{sec_zero_1}, which is the reason we preferred it.
\end{remark}

\begin{remark}
	In each cosphere bundle, the kernel of the canonical 1-form $\alpha = \xi \d x=\xi\partial_\xi \lrcorner \sigma$ defines a maximally non-integrable distribution, in other words a contact structure (see \cite[Appendix 4]{Arnold_Mech}).
	The proof above could also have been formulated using the language of contact geometry,  namely by defining instead of \eqref{Psi} the map $(y,\eta)=\hat \Psi(s,x',\xi')=\Phi^{s}(\zeta(x',\xi')):\R_{t}\times \mathcal{U}_{(x',\xi')}\to S^*M$ and showing that it satisfies $\Psi^{*}\alpha=\alpha'$, where the kernel of $\alpha'=\d s+\xi_j\d x^{j}$ defines a natural contact structure in the domain space. Using this, one would obtain a map between unit sphere bundles, which could then be extended by homogeneity to obtain the required result.
	Again, we did not take this approach because it does not appear to generalize in an obvious way to the zero energy case in Section \ref{sec_zero_1}. An inverse problem for contact forms and Reeb fields is addressed in \cite{KatzContact}.
\end{remark}

\begin{remark} 
So far we have not required $H(x,\xi)$ to be even in the fiber variable $\xi$. If it is, which is often the case with the applications we have in mind, then $(x(s),\xi(s))$ is a unit bicharacteristic if and only if $(x(-s),-\xi(-s))$ is, by Lemma~\ref{lemma_hom}. Then \r{A1} would be satisfied at $(x_0,-\xi^0)$ as well, and we can define $\zeta$ in \r{zeta} as $\zeta(x',-\xi') = -\zeta(x',\xi)$. Then $S(x,-\xi)$ is well defined (the parameter $s$ reverses sign along the ray). If $(y,\eta')=S(x,\xi')$, then $(y,-\eta')=S(x,-\xi')$, i.e., $S$ is odd in the fiber variable. Therefore, knowing $S$ on outgoing rays, in addition to the incoming ones (or vice-versa)  provides no additional information. In that case, $\kappa$ is odd as well. 
\end{remark}

\section{Linearization of the travel times} \label{sec_3}

\subsection{Travel times \texorpdfstring{$T(x,y)$}{T(x,y)}}
We will parameterize now the travel times by their initial and endpoints $x$ and $y$, writing $T(x,y)$. This requires a no-conjugacy assumption. Such travel times correspond to the actual time it takes a signal from $x\in \bo$ to arrive at $y\in\bo$ when the model is the hyperbolic PDE \r{a1}.

\begin{definition}
For $(x,\xi)\in T^*M\setminus 0$,   the exponential map is defined by $\exp_x \xi  = \Phi^{1}_x(x,\xi)$.
\end{definition}
 
\noindent In the metric case, this is the well known exponential map with $\xi$ associated with a vector by the musical isomorphism. 

\begin{definition}
	We say that $y=\exp_x \xi$ is not conjugate to $x$ along the  bicharacteristic issued from $\xi$, if $\exp_x$ is a local diffeomorphism near $\xi$.  If so, we call  $|\xi| \eqcolon T(x,y)$ the travel time between $x$ and $y$. 
\end{definition}

This is a generalization of the same notion in the metric case but there are some possible new effects. We do not a priori know that $x$ and $y$ are not conjugate if they are close enough. Also, note that $T(x,y)$ is defined as the travel time of the unique unit bicharacteristic from $x$ to $y$, when they are close enough to $x_0$ and $y_0$, respectively,  issued with codirection close enough to $\xi^0$, having travel time a priori close enough to $\hat \ell(x_0,\xi^{0})$. 

\subsection{A variational formula}
Assume that $H(x,\xi;s)$ is a one-parameter family of Hamiltonians near $s=0$ with travel times $T(x_1,x_2;s)$. If the Hessian $H_{\xi\xi}$ is positive, we can use the Legendre transform to get a Finsler metric (see Section \ref{sec_Finsler}), and use the fact that the Hamiltonian curves transform into lifts of geodesics, which are critical curves of the energy functional. For the fastest speed, the strict convexity is true in elasticity, at least, but in general it is not. 
 
The next lemma is an analog of the first variation of the energy functional in Riemannian geometry. There are two essential differences however. First, the perturbed curves are Hamiltonian, of some perturbed Hamiltonian function, not just any smooth variation. Second, there is a non-trivial integral term in \r{lt2} below, while in the analogous situation in metric case, where one perturbs the energy for curves near the lift of a geodesic to the tangent bundle, no such term appears. In fact, the total variation then vanishes if the endpoints are fixed, consistent with the well known fact that the geodesic we perturb is a critical point of the energy functional. Relation \eqref{lt2} shows in particular that Hamiltonian curves are not energy-minimizing among nearby curves.

\begin{lemma} \label{lemma_V}
Let $[0,1]\ni t\mapsto \gamma(t;s)\coloneqq  (x(t;s),\xi(t;s))$ be   Hamiltonian curves related to $H(x,\xi;s)$,  smoothly dependent on a real parameter $s$. 
Then the energy
\[
E(\gamma(\,\cdot\,;s)) \coloneq \int_0^{1} H(x(t;s),\xi(t;s);0)\,\d t
\]
satisfies 
\begin{equation}  \label{lt2}
	\frac{\d}{\d s}\Big|_{s=0} E(\gamma(\,\cdot\,;s)) = \langle \xi(t;0),  x_s(t;0) \rangle\Big|_{t=0}^{t=1} - 2  \int_0^1 H_s(x(t;0),\xi(t;0);0)\,\d t.
\end{equation}  

\end{lemma}

\begin{proof} 	
We have 
	\[
	  \frac{\p}{\p s} \Big|_{s=0} H(x(t;s),\xi(t;s);0) =  H_x\cdot x_s+ H_\xi\cdot \xi_s =   - \dot\xi\cdot x_s+ H_\xi\cdot \xi_s,
	\]
all restricted to $s=0$. 	
Integrate  with respect to $t$ from $0$ to $1$, and use integration by parts for the term $\dot\xi\cdot  x_s$. We get
	\begin{align}
		\frac{\d}{\d s}\Big|_{s=0} E(\gamma(\,\cdot\,;s))&=\int_0^1 (- \dot\xi \cdot x_s+ H_\xi\cdot \xi_s)\,\d t \\
		& =  %
		-\langle \xi(t;0), x_s(t;0)\rangle\Big|_{t=0}^{t=1} 
		+ \int_0^1 (\xi \cdot  \dot x_s+ H_\xi  \cdot \xi_s)\,\d t\\
		& =-\langle\xi(t;0), x_s(t;0)\rangle\Big|_{t=0}^{t=1} + \int_0^1 ( (\p_s H_{\xi})\cdot \xi  + H_\xi  \cdot \xi_s)\,\d t,\label{Hs1}
	\end{align}
where, again, all quantities are restricted to $s=0$. 	
	Here, $(\p_s H_{\xi})$ stands for the $s$-derivative of $H_\xi(x(t;s),\xi(t;s);s)$ with respect to all occurrences of $s$ at $s=0$, not just the last one. By the Euler identity,  the integrand in \r{Hs1} equals
	\[
	\frac{\p}{\p s}\big(H_\xi(x(t;s),\xi(t;s);s)\cdot\xi(t;s)\big) =  2 \p_s \big(H (x(t;s),\xi(t;s);s)  \big)
	\]
at $s=0$, and in particular, independent of $t$. This yields
	\begin{equation}  \label{lt1}
\frac{\d}{\d s}\Big|_{s=0} E(\gamma(\,\cdot\,;s)) = -\langle \xi(t;0),  x_s(t;0) \rangle\Big|_{t=0}^{t=1} + 2\frac{\p}{\p s}\Big|_{s=0} \int_0^1 H (x(t;s),\xi(t;s);s)\,\d t  .
\end{equation}  
For the latter term, we have
\begin{equation}\label{T22}
2\frac{\p}{\p s}\Big|_{s=0}\int_0^1 H (x(t;s),\xi(t;s);s) \,\d t
	=2 \frac{\d}{\d s}\Big|_{s=0} E(\gamma(\,\cdot\, ;s)) +  
	2\int_0^1 H_s(x(t;0),\xi(t;0);0)\,\d t. 
\end{equation}
Combining \eqref{lt1} and \eqref{T22} proves \r{lt2}.
\end{proof} 

\begin{example}
We compare the energy above with its counterpart in the metric case. The energy there is defined as an integral of $g_{ij}\dot \gamma^i \dot\gamma^j$ however. Converting this to the cotangent bundle, we have $ g^{ij} (g\dot\gamma)_i (g\dot\gamma)_j$, with $  g\dot\gamma=\xi$. Then one varies $\gamma$ by keeping $g$ the same. In the way we defined the energy in \r{lt1} however, in the metric case $\xi(t;s)$ is obtained from $\dot\gamma$ by lowering the index with $g(x;s)$, not with $g(x;0)$. Then the $s$-derivative would have additional terms. Using $\delta$ to denote variation, i.e., the $s$-derivative at $s=0$, %
we have 
\begin{align}
 \p_s|_{s=0}\left(H(x(t;s),\xi(t;s);0) \right) & = g^{ij} \delta ((g\dot\gamma)_i  (g\dot\gamma)_j)\\ 
 =  %
 2\big( g^{-1} (\delta g)\dot\gamma\big) &\cdot (g\dot\gamma) +  2 \big(g^{-1}(g \delta\dot\gamma )\big)\cdot g\dot \gamma
 = %
-2( (\delta g^{-1})(g\dot\gamma))\cdot (g\dot\gamma) + 2 \delta\dot\gamma \cdot g\dot \gamma. %
\end{align} 
Integrate this identity with respect to $t$ from $t=0$ to $t=1$. 
The first term gives us an integral of $ -2H_s (x(t;0),\xi(t;0);0)$. The second one is $\delta ((g(\,\cdot\,;0)\dot\gamma)\cdot\dot\gamma) $ for which we can apply the first variation formula of the energy of a geodesic to get the boundary terms in \r{Hs1}. Therefore, \r{Hs1} is confirmed in this case. 

Back to the general case, when the Legendre transform $\mathcal{L}$ is invertible (see section~\ref{sec_Finsler}), the curves solving the Euler-Lagrange equation on $TM$ minimize the action locally. They are critical curves of the action functional, which is how the Euler-Lagrange equation is derived in the first place.  The conversion to the energy defined on $T^*M$ above involves $\mathcal{L}_s$ depending on $s$, then being differentiated as $s=0$. This is the reason we cannot expect  the integral  \r{lt2} to be zero in general. 
\end{example}

\subsection{Linearization of \texorpdfstring{$T(x,y)$}{T(x,y)}}

\begin{corollary} \label{cor_T}
Assume that $x_1$ and $x_2 \in M$ 
are not conjugate along %
the bicharacteristic issued from $(x_1,\xi^1)$ related to $H(x,\xi,0)$. 
Let $[0,1]\ni t\mapsto (x(t;s),\xi(t;s))$ be   Hamiltonian curves related to $H(x,\xi;s)$, smoothly dependent on $s$ for $|s|\ll1$, such that their $x$-projections  connect $x_1$ and $x_2$. 	Then 
\begin{equation}  \label{lt}
 \frac{\p}{\p s}\Big|_{s=0} \frac12  T^2(x_1,x_2;s) =  - \int_0^{1} H_s(x(t;0),\xi(t;0);0)\,\d t.
	\end{equation}  
Parameterizing the bicharacteristics as unit ones, we get
\begin{equation}  \label{lt3}
 \frac{\p}{\p s} \Big|_{s=0} T(x_1,x_2;s) = - \int_0^{T(x_1,x_2;0)} H_s(x(t;0),\xi(t;0);0)\,\d t.
\end{equation}  	
\end{corollary}
 
\begin{proof}
Apply Lemma~\ref{lemma_hom} with $\lambda = T(x_1,x_2;s)$. Then the rescaled  curves in the current lemma are on energy levels $T^2(x_1,x_2;s)/2$ and they reach $x_2$ for $t=1$. 
	We have 	
\begin{equation}  \label{T21}
\frac12	T^2(x_1,x_2;s) = \int_0^1H(x(t;s),\xi(t;s);s)\,\d t
\end{equation}  
	in a trivial way because we are integrating the constant $T^2/2$ on the right.
We get \r{lt} from  \r{lt2}  and \r{lt1}.
  We make the change of the parameterization $\tilde t=Tt$. Then $ (x(\tilde t/T), T^{-1}\xi(\tilde t/T))$ is a Hamiltonian curve at energy level $1/2$. Performing that change of variables, we get \r{lt3}. 
\end{proof}

\begin{example}
Let $g$ be a Riemannian metric and let $H=\frac12 g^{ij}\xi_i\xi_j$. We have $H=\frac12 g_{ij}\dot\gamma^i\dot\gamma^j$ along the flow, where $\gamma$ are the geodesics. Energy level $1/2$ means unit speed geodesics $\gamma(t)$. Set $t=\tau T(x_1,x_2;s)$ and $\tilde\gamma(s) = \tilde \gamma(T(x_1,x_2;s)\tau)$.
Then the ``travel time'' is one in the $\tau$ parameter. We have $\frac{\d}{\d \tau} \tilde \gamma( \tau) = T(x_1,x_2;s) \dot\gamma (T(x_1,x_2;s) \tau)$, having length $T(x_1,x_2;s)$. The energy $E$ is given by
\[
E(\gamma(\,\cdot\,, s)) =\frac12 \int_0^1 g_{ij}(\tilde\gamma;s)  \dot{\tilde\gamma}^i(\tau;s) \dot{\tilde\gamma}^j(\tau;s)\,\d\tau = \frac12 T^2(x_1,x_2;s). 
\]
 Differentiate with respect to $s$ and use the minimizing property of the geodesics to get
\[
TT_s|_{s=0} = \frac12 \int_0^1 (\p_s g)_{ij}(\tilde\gamma)  \dot{\tilde\gamma}^i \dot{\tilde\gamma}^j|_{s=0}\,\d\tau.
\]
Going back to the original parameterization, we obtain
\begin{equation}   \label{dT}
\p_s T|_{s=0} = \frac12 \int_0^{T(x_1,x_2;0)} (\p_s g)_{ij}  \dot{\gamma}^i \dot{\gamma}^j|_{s=0}\,\d t,
\end{equation}
a well known formula. 
We will convert this to the cotangent bundle. The co-metric is $g^{-1}$, and its $s$-derivative is $\p_s g^{-1} = -g^{-1}(\p_s g) g^{-1}$. Then the variation $H_s$ of the Hamiltonian, written in the tangent bundle, becomes $H_s= -\frac12 (\p_s g)_{ij}  \dot{\gamma}^i \dot{\gamma}^j$. Then we get the same formula as in Corollary~\ref{cor_T}. This formula in the metric case is well known but it is usually written in the form \r{dT}, not in its covector version \r{lt}. 
\end{example}

\subsection{\texorpdfstring{$T(x,y)$}{T(x,y)} as a generating function of \texorpdfstring{$S$}{S}}
\begin{corollary}  \label{cor_dT}
Assume that $x_0$ and $y_0$ are not conjugate along   the bicharacteristic $\gamma_0$ issued from $(x_0,\xi^0)$, relative to $H(x,\xi)$.  Let 
$T(x,y)$ be the (locally defined) travel time between $x$ and $y$ in a neighborhood of $(x_0,y_0)$. Then the following holds near $(x_0,y_0)$. 
\begin{enumerate}[leftmargin=.8 cm,label=(\alph*)]
	\item \label{cordT:itema}
 
 $(y,\d_yT)= \Phi^{T(x,y)} (x,-\d_x T)$;

\item \label{cordT:itemb}
 We further have that 
\[
C \coloneq \left\{(x,- \d'_xT,y, \d_y'T) \right\}
\]
 is  the graph of the scattering relation $S$ ({relative to an incoming covector $(x_0,\xi^0)$ for which  \r{A1} holds),} in a neighborhood of $(x_0,\xi^0{}',y_0,\eta^0{}')$.  
 \end{enumerate}
\end{corollary} 

\begin{proof}
We apply Lemma~\ref{lemma_V}, where the Hamiltonian is independent of the parameter $s$ and the curves depend on $s$ only   
through their endpoints. Then $H_s=0$ in \r{lt2}, and the variation of the energy relation looks similar to its metric equivalent. 
Let $x(s)$, $y(s)$ be smooth curves in $M$
for $|s|\ll1$ with $x(0)=x_0$, $y(0)=y_0$. Let $(q(t;s), p(t;s))$ be the  smooth Hamiltonian curves with $q(0,s)= x(s)$, $q(1,s)= y(s)$. Set $p(0,0)= \xi^0$,  $p(1,0)=\eta^0$, $v=q_s(0;0)$, $w=q_s(1,0)$. Then, see \r{T21}, 
\begin{equation}  \label{T2}
\p_s \frac12 T^2(x(s),y(s))|_{s=0} =  \langle \eta^0,w\rangle- \langle \xi^0,v\rangle.
\end{equation}  
The derivative on the left equals $T \langle \d_x T, v\rangle + T \langle \d_y T, w\rangle$ 
at $s=0$. Divide by $T$ to get
\[
\p_s  T(x(s),y(s))|_{s=0} =  \langle \tilde \eta^0,w\rangle- \langle\tilde  \xi^0,v\rangle,
\]	
where $\tilde \xi^0$ and $\tilde\eta^0$ are normalized now to unit ones. Since this is true for every two vectors $v$ (at $x_0$), and $w$ (at $y_0$), this proves \ref{cordT:itema}. 

Part \ref{cordT:itemb} follows immediately from \ref{cordT:itema}.
\end{proof}

\subsection{The X-ray transform} 
We define the X-ray transform 
\begin{equation}  \label{X}
\X f(\gamma) = \int f\circ\gamma(t)\,\d t = \int f(x(t),\xi(t))\, \d t, \quad f\in C^\infty(S^*M)
\end{equation}  
over all unit maximal Hamiltonian curves $\gamma$ through $S^*M$, assuming they are defined for finite time. We can also define $\X f$ over not-necessarily unit curves by assuming that $f(x,\xi)$ is homogeneous of order $m=0,1,\dots$ in $\xi$; then  $\X f(x,\xi)$ is homogeneous of order $m-1$ when parameterized by initial points. 
The transform $\X$ is a linearization of the travel times, as Corollary~\ref{cor_T}, equation \r{lt3} shows. There, $f=H_s|_{s=0}$ is a priori homogeneous of order two  but since it is restricted to $H=1/2$, the homogeneity is irrelevant. 

The transform $\X$ is easy to invert modulo a gauge. 

\begin{theorem}[Inversion of $\X$] \label{thm_X1}
Let $H$, $\mathcal{U}$, $\mathcal{V}$, and $\Gamma$ be as in the discussion preceding Theorem~\ref{thm_H}. Then for every $f\in C^\infty(S^*M\cap \Gamma)$, $\X f=0$ for all maximal unit bicharacteristics in $\Gamma$  if and only if $f=X_H\phi$  with some $\phi\in C^\infty(S^*M\cap \Gamma)$ so that $\phi=0$ on $\Gamma_{U}\cup\Gamma_{V}$. 
\end{theorem}
\begin{proof} 
The ``if'' part is obvious. For the ``only if'' statement, set $\phi(x,\xi) = \int_{-\ell_{-}(x,\xi)}^0 f(\Phi^s(x,\xi))\,\d s$. Clearly, $\phi=0$ on $\Gamma_{U}$. Since $\X f=0$, $\phi=0$ on $\Gamma_{V}$. 
\end{proof}

We show now that linearizing a family of gauge equivalent Hamiltonians we obtain the natural nullspace of $\X$. Let $\kappa(\,\cdot\, ,\epsilon) = \Id + \epsilon F+O(\epsilon^2) $, $|\epsilon|\ll1$ be a family of homogeneous canonical transformations satisfying $\pi\circ \kappa(\Gamma_U),\pi\circ\kappa(\Gamma_V)\subset \bo$ and $\iota^{*}\circ \kappa=\iota^{*}$ on $\Gamma_U\cup \Gamma_V$. 
We may assume that we are working in a small enough neighborhood of our reference bicharacteristic that $\Gamma_U$ is contractible; then $\Gamma $ is also contractible.
The infinitesimal canonical transformation  $F$ is of the form $F= X_\phi$ for some $\phi\in C^\infty(\Gamma)$, uniquely determined up to constant.
Indeed, since $\kappa(\cdot, \epsilon)$ are symplectic, by  Cartan's formula
$\d (F\lrcorner\sigma) =\L_F\sigma=0$ ($\L$ is the Lie derivative), and because $\Gamma$ is contractible, $F\lrcorner\sigma=-\d \phi$ for some $\phi$, that is, $F=X_\phi$.
Further, by homogeneity of $\kappa$ we have 
\begin{equation}
	\xi \phi_{\xi \xi }=0,\quad \xi \phi_{x \xi }=\phi_x
	\implies d(\phi-\xi\phi_\xi)=0\text{ on }\Gamma		.\label{eq:homog_taut}
\end{equation} %

We now check that $\phi=0$ on $\Gamma_U\cup \Gamma_V$ (up to constant).
Choose a point $\zeta_0=(x_0,\xi^0)=((x',0),(\xi',\xi_n))\in \Gamma_U$, written in boundary coordinates as usual.
Then write $F(\zeta_0)=a\cdot \partial _x+b\cdot \partial_\xi$, with $a=\partial_\xi \phi\big|_{\zeta_0}$, $b=-\partial_x \phi\big|_{\zeta_0}$.
Since for all $|\epsilon|\ll 1$ we have $\pi\circ  \kappa(\zeta_0,\epsilon)\in U$ and $\iota^{*}\circ \kappa(\cdot,\epsilon)=\iota^{*} $, we see that $a=0$ and $b'=-\partial_{x'}\phi\big|_{\zeta_0}=0$.
Since $\zeta_0\in \Gamma_U$ was arbitrary, we see that $\phi\big|_{\Gamma_U}$ is locally constant, hence also constant there and can be taken to be 0.
Taking it to be 0 on $\Gamma_U$ we have $\phi-\xi \phi_{\xi}=0$ on $\Gamma_U$, thus $ \phi=\xi \phi_{\xi}$ on $\Gamma$ by \eqref{eq:homog_taut}, that is, $\phi$ is homogeneous of degree 1 on $\Gamma$.
An argument	 as before shows that $\phi$ is also constant on ${\Gamma_V}$, so by homogeneity it has to vanish there.

Linearizing the Hamiltonians, we find
\begin{equation}\label{eq:linearization}
	\begin{aligned}
(\kappa^{-1} (\,\cdot\, ,\epsilon))^*H(x,\xi)& = H\big((x,\xi)-\epsilon F(x,\xi) \big)+O(\epsilon^2)\\
& = H(x,\xi) - \epsilon \d H\cdot F+O(\epsilon^2) = H(x,\xi) + \epsilon X_H\phi +O(\epsilon^2).
\end{aligned}
\end{equation}
Thus the linearization is $X_H\phi$ indeed, with $\phi$  homogeneous of degree 1 vanishing at $\Gamma_U\cup \Gamma_V$.

We finally mention that if $\kappa (\,\cdot\, ,\epsilon)=(f(\,\cdot\, ,\epsilon),\d (f^{-1})^{*}(\,\cdot\, ,\epsilon)) $ for a family $f(\,\cdot\, ,\epsilon)$ of diffeomorphisms on $M$ fixing $\bo$ pointwise, in which case the linearization of $f(\,\cdot\, ,\epsilon)$ is a vector field $V$ on $M$ which vanishes at $\bo$, one checks that $F=(V,-\d_x(\xi\cdot V))$, so the potential above is given by $\phi=\alpha (V)$, up to constant.
If in addition $H=\frac12 g^{ij}\xi_i\xi_j$ one checks, for instance using local coordinates, that the linear term  in \eqref{eq:linearization} is given by $X_H\phi = \xi(\nabla_{\xi^\sharp} V)=(\d ^{\mathrm s} V^\flat)(\xi^\sharp,\xi^\sharp)$, where the covariant derivative $g$,  the symmetric differential $\d ^{\mathrm s}$, and the musical isomorphisms are with respect to $g$. 
This is the well known expression for the linearization of a 1-parameter family of metrics, transferred to the phase space.

\section{Zero energy level: Scattering rigidity for operators of real principal type} \label{sec_zero}

\subsection{Setup} \label{sec_zero_1}
Let $p_m(x,\xi)\in C^\infty(T^*M\set)$ be real and homogeneous of order $m$ in $\xi$. 
Let $p$ be a symbol in $S^m$ with principal symbol $p_m$. 
Assuming $\partial M =\emptyset$, H\"ormander's propagation of singularities theorem says that $\WF(u)\setminus\WF(f)$ of the solutions of $p(x,D)u=f$ is contained in $p_m^{-1}(0)$ and invariant under the Hamiltonian flow of $p_m$ (at energy level $p_m=0$). There is no time-space separation of the variables here. The parameterization of the zero bicharacteristics is  irrelevant; they are considered as point sets in the theorem. In fact, multiplying $p_m$ by an elliptic symbol (on $p_m=0$) changes the parameterization only, as we show below in Lemma~\ref{lemma_mu} (and which is known). 
This makes sense, because it amounts to applying an elliptic \PDO\ to our equation. This motivates our interest in studying the scattering relation for such homogeneous Hamiltonians $H\in C^\infty(T^*M\set)$. We assume $m=2$ (which can be converted to any other order by a multiplication by an elliptic factor),  and that $\d H\not=0$ on $H^{-1}(\{0\})$. Moreover, we assume that $H^{-1}(\{0\})\neq\emptyset$ (the case $H^{-1}(\{0\})=\emptyset$ corresponds to elliptic operators and is not interesting from our point of view).  

The inverse problem for such operators was studied in \cite{OSSU-principal}. It was shown there that the Cauchy data on the boundary determines uniquely the scattering relation $S_0$ (defined below) and the zero bicharacteristic transform $L$ (see \r{L1}) of the terms of order one or two lower. The inversion of those transforms was left as an open problem. 

\begin{lemma}[{\cite{S-Lorentz_analytic}}] \label{lemma_mu}
	Let $M$ be a manifold with or without boundary. 
	Let $H(x,\xi)$ be a Hamiltonian defined near $(x_0,\xi^0)\in T^*M\set$, and assume $ H(x_0,\xi^0)=0$, $\d H(x_0,\xi^0)\not=0$. Let $\mu(x,\xi)>0$ near $(x_0,\xi^0)$. 
\begin{enumerate}[leftmargin=.8 cm,label=(\alph*)]
	\item \label{lemmamu:itema}
	Then the Hamiltonian curves of $H$ and $\breve{H}\coloneqq  \mu H$
near $(x_0,\xi^0)$ on the zero energy level coincide as point sets but have possibly different parameterizations. 
More specifically, if $(x(s),\xi(s))$, and $(\breve{x}(\breve{s}), \breve{\xi}(\breve{s}))$ are solutions related to $H$ and $ \breve H$, with initial conditions $(x_0,\xi^0)$ 
	at $s=0$, $\breve{s}=0$, respectively, then 
	\begin{equation}  \label{xs}
	(\breve{x}(\breve{s}(s)), \breve {\xi}(\breve s(s))) = (x(s),\xi(s))
	\end{equation}  
	with $\breve s(s)$ solving
	\begin{equation}  \label{ss}
	\frac{\d \breve s}{\d  s}=\frac{1}{\mu( x( s),\xi(s))}, \quad \breve s(0)= 0.
	\end{equation}

\item \label{lemmamu:itemb} Let $\breve s(s; x_0,\xi^0)$ be such that $\d\breve  s/\d s>0$, and define $\mu$ on the Hamiltonian trajectory through $(x_0,\xi^0)$ by \r{ss}. Then  \r{xs} holds with the trajectory $ (\breve x, \breve\xi)$ determined by $\breve{H} \coloneq \mu H$ for any smooth extension of $\mu$. 
\end{enumerate}
\end{lemma}
\begin{proof}Part \ref{lemmamu:itema} is proved in \cite{S-Lorentz_analytic}. The proof of \ref{lemmamu:itemb} follows from that of \ref{lemmamu:itema} there. 
\end{proof}

An example is $p=g^{ij}(x)\xi_i \xi_j/2$, with $g$ Lorentzian. Any positive conformal factor reparameterizes the zero bicharacteristics but keeps them the same as point sets. There is no natural way to define travel times here unless one chooses a time function (in a non-unique way). Note that our assumptions also allow $g$ to be pseudo-Riemannian of any signature $(p,q)$, as long as $pq \neq 0$. 
The scaling Lemma~\ref{lemma_hom}, with $E=0$ now, confirms those arguments.

\begin{remark}\label{rem_zero}

Two possibly different Hamiltonians $\breve H$ and $H$ with the same zero energy hypersurface $\mathcal H$ satisfy $\breve H= \mu H $ 
with some non-vanishing $\mu\in C^{\infty}(T^*M\set)$ which is homogeneous of order 0.
Thus by Lemma~\ref{lemma_mu}, the Hamilton flows of $\breve H$ and  $H$ on $\mathcal H $ agree up to reparameterization, which by \eqref{ss} depends only on the values of $\mu$ on $\mathcal H $.
From another point of view, the Hamiltonian vector field $X_H$ is determined on $\mathcal H = H^{-1}(\{0\})$ once one knows the differential $\d H$ there. The differentials of two different Hamiltonians $H$ and $\breve H$ with zero energy hypersurface $\mathcal H$ are related there by $\d\breve H=\mu \d H $ for some scalar function $\mu\in C^\infty(\mathcal H)$ (since they are both conormal) and on $\mathcal H$, $X_{\breve H}= \mu X_{ H} $. 
 From either point of view, we see that the zero Hamiltonian flow depends on $\mathcal H$ (any defining function of which determines the bicharacteristics as point sets) and on a conformal factor on $\mathcal{H}$ (which fixes their parameterization).
The values of the conformal factor $\mu$ on $\mathcal{H}$ are course the same, regardless of which point of view one takes. 

This is markedly different from the case of a positive, homogeneous Hamiltonian at positive energy level: then a positive energy hypersurface is not conic, and in fact for any $\xi\neq 0$ the radial ray $s\mapsto s\xi$ intersects $\{ H=E\}$ transversely.
This implies that $\{\breve H=E\} = \{H=E\}$ is equivalent to $\breve H=H$, since $H$ (hence also its flow) is determined globally by $\{ H=E\}$ using homogeneity. 
\end{remark}

Assuming $\det H_{\xi\xi}\not=0$, for fixed $x$ and $E>0$, the image of  $\exp_x(t\xi)$, for $\xi$ in an open set with  $|\xi|=\sqrt{2E}$ and for $t$ in an open set about $0$,
 is open again. When $E=0$, the domain is contained a conic set of codimension one, and the image is diffeomorphic to it when $\det H_{\xi\xi}\not=0$, at least before it hits conjugate points. Basic examples are $x=x_0+s\theta$, $|\theta|=1$ when $H=\frac12|\xi|^2/2$, $E=1/2$; and $x=x_0+s(1,\theta)$ when $H=\frac12 g^{ij}\xi_j\xi_j$
 with $g$ Minkowski. In the latter case, the image is the light cone $(x^0)^2= (x^1)^2+\dots+(x^n)^2$. This observation is important when $H(x,\xi)$ has no special structure in the $x$ variable, but it does in the $\xi$ one, like being a polynomial of $\xi$ or an eigenvalue of such a matrix. 

\subsection{Scattering rigidity}\label{sec_zero_S}
Similarly to section~\ref{sec_2}, we start with a zero  bicharacteristic curve $\gamma_0$ in a manifold with boundary $M$ connecting some distinct covectors $(x_0,\xi^0)$ and $(y_0,\eta^0)$ in $T^*M\set\big|_{\partial M}$ and contained in $\{H=0\}$ for a fixed Hamiltonian $H$. Choose two disjoint open subsets $U\ni x_0$ and $V\ni y_0$ of $\partial M$ 
so that \r{A1} holds at each one of them. 
By our assumption, $\ell_\pm(x,\xi)<\infty $ along $\gamma_0$ ($\ell_\pm$ are as in \eqref{eq:ell_pm} but at energy level 0), so by \eqref{A1} the same holds in a neighborhood of $(x_0,\xi^{0})$ in $ H^{-1}(\{0\})$ and $\ell_\pm$ is smooth there.
One can define the scattering relation $\hat{S}_0(x,\xi)=\Phi^{\ell_+(x,\xi)}(x,\xi)$ for incoming covectors, as in \eqref{a2} (with the subscript zero indicating a zero energy level), and it is homogeneous of degree 1. 
To define the scattering relation $S_0$ for covectors in $T^*U,$ we restrict $(x,\xi)$ and $(y,\eta)=\hat{S}_0(x,\xi)$ to $TU$ and $TV$ as before to obtain respectively elements $(x,\xi')\in T^*U$ and $(y,\eta')=S_0(x,\xi')\in T^*V$. 
By the Implicit Function Theorem, one can solve $H(x,\xi)=0$ with respect to $\xi$  for given $(x,\xi')$ in a sufficiently small conic neighborhood $\mathcal{U}\subset T^*U$ of $(x_0,\xi^0{}')$,   and similarly in a conic neighborhood $\mathcal{V}$ of $ (y_0,\eta^0{}')=S_0(x^0,\xi^0{}')$.
The new moment now is that $\mathcal{U}$ and $\mathcal{V}$ are conic, and $S_0$  is homogeneous of order one. 
When $H=\frac12g^{ij}\xi_i\xi_j$ with $g$ Lorentzian and $U$ is timelike, $\mathcal{U}$ can be any conic open subset of the timelike cone $g(\xi',\xi')<0$ on $T^*U$. When $U$ is spacelike, then $\mathcal{U}$ can be any conic open subset of $T^*U$.
There is no natural choice of travel times now. %
Relative to our fixed $(x_0,\xi^0)$ we set
\[
\begin{split}
\Gamma_0\coloneqq  \big\{(y,\eta)|\;& \text{$(y,\eta)$ lies on the zero bicharacteristic } \text{issued from some $(x,\xi)$ with $(x,\xi')\in\mathcal{U}$}\}.
\end{split}
\]
We assume, upon shrinking $\Gamma_0$ if necessary, that for  $(x,\xi)\in \Gamma_0\big|_U=\{(x,\xi)\in \Gamma_0:x\in U \}$, the equation $H(x,\xi)=(x,\xi')$ is uniquely solvable for $(x,\xi)$, and we set $(x,\xi)=\zeta(x,\xi') $.
By shrinking $\Gamma_0$ further if needed, we can assume also that on $\Gamma_0$ we have $\ell_\pm<\infty$ and that $\ell_\pm$ are smooth.
We denote by $\sigma_0$ the restriction of the symplectic form on $T^*M$ to $\Gamma_0$.

Below, $\tilde M$ is a manifold with boundary with the same dimension as $M$ and the same boundary, and with $\tilde H\in C^\infty (T^*\tilde M\set )$ a homogeneous Hamiltonian having the same properties as $H$. All quantities, sets, etc corresponding to $(\tilde M,\tilde H)$  will be denoted with tildes, and on them we will have the same assumptions as for the analogous ones corresponding to $(M,H)$.

\begin{theorem}[Semi-global rigidity at a zero energy level]  \label{thm_0_maybe}\

\begin{enumerate}
	[leftmargin=.8 cm,label=(\alph*)]

\item\label{zero_en:itemb} 
Suppose that there exists a smooth, homogeneous diffeomorphism $\kappa: \Gamma_0\to \tilde \Gamma_0$ satisfying $\kappa^{*}\tilde \sigma_0=\sigma_0$ and \r{kappa_b}.
Then the scattering relations of $ H$ and $\tilde H$ agree on $\mathcal{U}$. 

\item\label{zero_en:itemc} 

Assume that $S_0^{H}= S_0^{\tilde H}$ on $\mathcal{U}$. Then there exists a positive $\mu\in C^\infty(\Gamma_{0})$, homogeneous of degree 0, such that for the Hamiltonian $\breve H\coloneqq \mu H$ one has $\breve \ell_+ \circ \zeta =\tilde \ell_+\circ \tilde \zeta$ on $\mathcal U$.
If $\mu'$ is another function with the same properties as $\mu$, then $\int  (\mu^{-1}-\mu'^{-1})\circ \gamma(t)\d t=0$ for null every bicharacteristic $\gamma$ contained in $\Gamma_0$.
Moreover, there exists a unique homogeneous diffeomorphism $\kappa: \Gamma_0\to \tilde \Gamma_0$ with $\kappa^{*}\tilde \sigma_0=\sigma_0$ satisfying \r{kappa_b} and conjugating the flows of $\breve H $ and $\tilde H$, that is, 
\begin{equation}\label{eq:conjugation_of_flows}
	 \kappa\circ \breve \Phi^t(x,\xi)=\tilde \Phi^t\circ \kappa(x,\xi)\quad \forall (x,\xi)\in \Gamma_0,\quad t\in [-\breve\ell_-(x,\xi),\breve\ell_+(x,\xi)].
\end{equation}
\end{enumerate}
\end{theorem}

\begin{remark}
Note that by taking the identity in \ref{zero_en:itemb}, if $\Gamma_0=\tilde \Gamma_0$, the scattering relations of $H$ and $\tilde H$ agree. 
In particular, $H$ and $\mu H$ have the same data for any (homogeneous) $\mu\not=0$ on $H^{-1}(\{0\})$. 
Moreover, part \ref{zero_en:itemb} implies that if a Hamiltonian $H$ is the pullback of $\tilde H$ by a homogeneous canonical transformation satisfying \eqref{kappa_b} up to multiplication by conformal factor, then $H$ and $\tilde H$ have the same scattering relations. 
This is the situation one expects in the metric case.
If two Lorentzian metrics on $M$ are related by $g=\mu \psi^* \tilde g$, where $\psi$ is  a diffeomorphism of $M$ fixing $\bo$ and $\mu$ is a conformal factor,  we automatically obtain a canonical transformation of the whole phase space relating the corresponding Hamiltonians up to conformal factor which only depends on the base variables.

In the inverse problem, one could extend $\kappa$ to a diffeomorphism 
$K$
defined in a neighborhood of $\Gamma_0$ (and one could even make it canonical, for example using the coisotropic embedding theorem, see \cite{CannasDaSilva}).
Then $K^\ast\tilde H = \tilde \mu H$ near $\Gamma_0$ for some extension $\tilde \mu$ of $\mu$, whose existence is guaranteed by the fact that $H$ and $\tilde H$ are defining functions of the same hypersurface.
However, these extensions are not unique, and  our data only allows us to construct a diffeomorphism between the codimension one hypersurfaces $\Gamma_0$ and $\tilde {\Gamma}_0$, without providing us with a preferred way to extend it to a neighborhood.
\end{remark}

\begin{proof}[Proof of Theorem \ref{thm_0_maybe}]
For part \ref{zero_en:itemb}, we will  show that for every $\zeta_0\in \Gamma_0\big|_U$ there exists a smooth, strictly increasing function $\alpha: [0, \ell_+(\zeta_0)]\to [0,\infty)$ such that for $t\in [0, \ell_+(\zeta_0)]$ one has 
\begin{equation}\label{flow_intertwiner}
	\kappa\circ \Phi^t(\zeta_0)=\tilde \Phi^{\alpha(t)}\circ\kappa(\zeta_0). 
\end{equation}
First observe that $\Gamma_0$ is coisotropic with respect to $\sigma $, as a codimension 1 submanifold of a symplectic manifold. 
Hence for each $\zeta\in \Gamma_0$, the 1-dimensional symplectic complement of 
$T_\zeta\Gamma_0$ with respect to $ \sigma$ is contained in $T_\zeta\Gamma_0$, and one sees that it is actually spanned by $X_H\big|_{\zeta}$.
 A similar discussion holds for $\tilde \Gamma_0$. 
Given $\tilde \zeta\in \tilde \Gamma_0$ and $v\in T_{\tilde \zeta}\tilde \Gamma_0$, and using that $\kappa^{*}\tilde \sigma _0=\sigma_0$,
\[\tilde \sigma (\d \kappa X_H,v)=\tilde \sigma_0 (\d \kappa X_H,v)= \sigma_0 ( X_H,\d \kappa^{-1}v)= \sigma ( X_H,\d \kappa^{-1}v)=0,\]
that is, $\d \kappa X_H\big| _{\tilde \zeta}$ is in the symplectic complement of $ T_{\tilde\zeta}\tilde\Gamma_0$ and therefore it must be a scalar multiple of $X_{\tilde H}\big|_{\tilde \zeta}$.
Since $\d \kappa X_H$ and $X_{\tilde H}$ are both smooth and non-vanishing, there exists a non-vanishing function $\mu\in C^\infty(\tilde \Gamma_0)$ such that $\d \kappa X_H = \mu X_{\tilde H}$.
Since by \eqref{kappa_b} $\kappa$ takes incoming covectors to incoming ones, $\mu>0$ everywhere.
Hence the flows of $\d \kappa X_H$ and $X_{\tilde H} $ are agree up to orientation preserving reparametrization, implying the existence of  $\alpha $ satisfying \eqref{flow_intertwiner}.

Using \eqref{kappa_b} and \eqref{flow_intertwiner} we see that $\alpha(\ell_+(x,\xi))=\tilde \ell_+(\kappa(x,\xi))$, which leads us to 
 \begin{equation}\label{incoming_scat}
 	\kappa \circ \hat S_0^{ H}(x,\xi)= \hat S_0^{\tilde H} \circ\kappa(x,\xi), \quad (x,\xi)\in \Gamma_{0}\big|_U
. \end{equation}
To pass to the scattering relations on $\mathcal{U}$, 
we use that $ S^{ H}=\iota^* \circ \hat S_0^{{H}}\circ \zeta$, and similarly for $\tilde H$.
Recalling that $\iota^{*}\circ \kappa (x,\xi)\overset{\eqref{kappa_b}}=(x,\xi')$ and  $\kappa (x,\xi)=\tilde \zeta(x,\xi')$,
we find \begin{align}
	S^{ H}(x,\xi')=\iota^{*} \circ \hat{S}^{ H}_0\circ \zeta(x,\xi')\overset{\eqref{kappa_b}}=
	\iota^{*}\circ\kappa \circ \hat{S}^{H}_0(\zeta(x,\xi'))
	\overset{\eqref{incoming_scat}}=\iota^{*}\circ \hat S^{\tilde{H}}_0\circ \kappa(\zeta(x,\xi'))\\
	=\iota^{*}\circ \hat S^{\tilde{H}}_0\circ\tilde\zeta(\iota^{*}\circ  \kappa(\zeta(x,\xi')))=S_{0}^{\tilde H}(\iota^{*}(\zeta(x,\xi')))=S_{0}^{\tilde H}(x,\xi').
\end{align}
Thus the scattering relations of $\tilde H$ and $ H$ agree on $\mathcal U$.

 \smallskip

For part \ref{zero_en:itemc}, to construct $\mu$, extend by constancy along the flow of $X_H$ the function $\Gamma_0\big|_U\ni (x,\xi)\to \frac{\ell_+(x,\xi)}{\tilde\ell_+(\tilde \zeta\circ \iota^{*}(x,\xi))}\in (0,\infty)$.
In other words, for $(y,\eta)\in \Gamma_0$ we set 
$
	\mu(y,\eta)\coloneqq \frac{\ell_+(\Phi^{-\ell_-(y,\eta)}(y,\eta))}{\tilde\ell_+\big(\tilde \zeta\circ \iota^{*}(\Phi^{-\ell_-(y,\eta)}(y,\eta))\big)}.
$
This is homogeneous of degree $0$ and satisfies   $X_H\mu=0$ by construction.
By Lemma \ref{lemma_mu}, for the flow $\breve \Phi^{\breve s(s)}$ of $\breve{H}\coloneqq\mu H$ we have $\breve \Phi^{\breve s(s)}=\Phi^s$, where by \eqref{ss} we have 
\begin{equation}
	\breve {s}(s;x,\xi)=\frac{\tilde\ell_+\big(\tilde \zeta\circ \iota^{*}(\Phi^{-\ell_-(x,\xi)}(x,\xi))\big)}{\ell_+(\Phi^{-\ell_-(x,\xi)}(x,\xi))}s\quad \text{ for }(x,\xi)\in \Gamma_0. 	
\end{equation}
Substituting $s= \ell_+(x,\xi)$ for some $(x,\xi)\in \Gamma_0\big|_U$ (so that $\ell_-(x,\xi)=0$), we see that $\breve \Phi^{\tilde \ell_+(\tilde \zeta\circ \iota^* (x,\xi))}\in \Gamma_0\big|_V$, 
from which it follows that 
\begin{equation}\label{eq:travel_times}
	\breve \ell_+(x,\xi)=\breve{s}(\ell_+(x,\xi);x,\xi)=\tilde \ell_+(\tilde \zeta\circ \iota^* (x,\xi)).
\end{equation}
This is equivalent to $\breve \ell_+\circ \zeta = \tilde \ell_+\circ \tilde \zeta$ on $\mathcal{U}$.
If $\mu'$ is another function with the same properties as $\mu$, it follows by  \eqref{ss} that it must satisfy 
\begin{equation}
	\int_0^{\ell_+(x,\xi)}\frac{1}{\mu(\Phi^{s}(x,\xi))}\d s=\tilde \ell_+(\tilde \zeta\circ \iota^* (x,\xi))=\int_0^{\ell_+(x,\xi)}\frac{1}{\mu'(\Phi^{s}(x,\xi))}\d s,
\end{equation}
for $(x,\xi)\in \Gamma_0\big|_0$. From this one also sees that if  $X_H\mu=X_H\mu'=0$, they must agree.

Similarly to the proof of Theorem \ref{thm_H} we
 construct  $\kappa:\Gamma_0\to\tilde \Gamma_0$ by taking 
\begin{equation}\label{eq:kappa_def_0}
 	\kappa(y,\eta)=\tilde \Phi^{\breve \ell_-(y,\eta)}\circ \tilde \zeta\circ \iota^*\circ \breve\Phi ^{-\breve \ell_-(y,\eta)}(y,\eta) .
 \end{equation}
The smoothness and homogeneity follow from the remarks in the beginning of section \ref{sec_zero_S}.
The properties \eqref{kappa_b} follow using \eqref{eq:travel_times}, and \eqref{eq:conjugation_of_flows} can be checked using that $\breve \ell_-( \breve \Phi^{t}(y,\eta))=t+\breve \ell_-(y,\eta)$.
Regarding uniqueness, \eqref{eq:conjugation_of_flows} for $t=\breve\ell_-(y,\eta)$ implies that $\kappa$ is uniquely determined by its values on $\Gamma_0\big|_U$. There, \eqref{kappa_b} implies that we must have $\kappa=\tilde \zeta\circ \iota^{*}$.

We now show that $\kappa^{*}\tilde \sigma_0=\sigma_0$.
As usual, we extend the manifold $M$ to a slightly larger manifold $N$ of the same dimension and $\breve H$ to a neighborhood of $T^*M\set$ in $T^*N\set$.
Fix $(y_1,\eta^1)\in \Gamma_0$ and let $(x_1,\xi^1)=\breve\Phi^{-s_0}(y_1,\eta^1)\in \Gamma_0$, where $s_0\coloneqq \breve\ell_-(y_1,\eta^1)$.
Define the map 
\begin{equation}
	 \breve \Psi(x',s,\xi',E )=\breve\Phi^{s}(x',0,\xi',\breve h(x',\xi',E))
\end{equation}
exactly as in \eqref{Psi}, except now $s\in (s_0-\epsilon,s_0+\epsilon)$, where $0<\epsilon\ll 1$ and $\breve h$ is defined implicitly by $\breve H(x',0,\xi',\breve h)=E$ for $E$ near $0$.
Note that $(x',0,\xi',\breve h(x',\xi,0))=\zeta(x',0,\xi')$ and $\breve\Psi(x_1{}',s_0,\xi^1{}',0)=(y_1,\eta^{1})$. By 
\eqref{eq:matrix} and \eqref{eq:symplectic_psi}, $\d\breve\Psi\big|_{s=s_0}^*\tilde\sigma=\sigma$
and the map is locally invertible. Constructing a similar map $\tilde {\Psi}$ corresponding to $\tilde H$, which is symplectic for the same reasons, we have $\kappa (y_1,\eta^1)=\tilde \Psi(x_1{}',s_0,\xi^1{}',0)=\tilde \Psi\circ \breve\Psi^{-1}(y_1,\eta^1)$ and $\d( \tilde\Psi\circ \breve\Psi^{-1})=\d \kappa$ on $T\Gamma_0$. 
Since $ \big(\d (\tilde \Psi\circ \breve\Psi^{-1})\big|_{(y_1,\eta^1)}\big)^{*}\tilde \sigma=\sigma$, $\kappa^{*}\tilde \sigma_0=\sigma_0$.
 \end{proof}

\subsection{Linearization}
There are no natural travel times in this case but one can define defining functions of pairs of points serving as initial and endpoints of null Hamiltonian curves, assuming no conjugacy. The second author showed in \cite{S-Lorentzian-scattering} that in the Lorentzian case, any such function is a generating function of the scattering relation, and its linearization is the light ray transform. We prove an analogous result in this case. 
We assume
\begin{equation}   \label{A2}
\text{$y_0\in V$ is not conjugate to $x_0\in U$ along $\gamma_0$.}
\end{equation}
Then for $(x,y)\in U\times V$ close enough to $(x_0,y_0)$, there is unique (not necessarily null) bicharacteristic close enough to $\gamma_0$. We shrink $U$ and $V$ if needed, to ensure that property. 

\begin{definition}
Assume \r{A2}. 
\begin{enumerate}[leftmargin=.8 cm,label=(\alph*)]
 \item  The set $\Sigma\subset U\times V$ consists of pairs $(x,y)$ so that $x$ and $y$ are connected by a locally unique zero bicharacteristic. 

\item We call $r(x,y)$  a defining function of $\Sigma$ on $U\times V$ if (i) $\Sigma= \{r=0\}$, (ii) $\d r\not=0$ on $\Sigma$, and (iii) $r>0$ if and only if the locally unique Hamiltonian curve connecting $x$ and $y$ is on a positive energy level. 
\end{enumerate}
\end{definition}

Condition (iii) is just a sign convention, consistent with that in \cite{S-Lorentzian-scattering}. 

\begin{theorem}
With the assumptions above, including \r{A2}, 
\begin{enumerate}[leftmargin=.8 cm,label=(\alph*)]
\item\label{Linearization_0_itema} $\Sigma$ is a codimension one  submanifold  of $U\times V$, assuming that $U$ and $V$ are small enough. 

\item\label{Linearization_0_itemb} Let $r(x,y)$ be a defining function of $\Sigma$ on $U\times V$.   Then 
\begin{equation}  \label{Sigma}
C = \{(x,-\lambda \d'_x r, y,\lambda \d_y'r);\; \lambda >0, (x,y)\in\Sigma\}
\end{equation}  
is the graph of $S_0$, locally near $(x_0,\xi^0,y_0,\eta^0)$. 

\item\label{Linearization_0_itemc}  If $H(\,\cdot\,, \,\cdot\,  ;s)$ is a one-parameter family of Hamiltonians satisfying the assumptions above, smoothly depending on $s$ near $s=0$, and $r_s$ is smooth family of   defining functions of the corresponding $\Sigma(s)$, then
\begin{equation}  \label{dr}
\frac{\p}{\p s}\Big|_{s=0} r_s(x,y) = -\rho(x,y) \int_0^1 f\circ\gamma_{[x,y]}(t)\,\d t \quad \text{on $\Sigma=\{r(x,y)=0\}$},
\end{equation}  
where $f=\d H(\,\cdot\,, \,\cdot\,  ;s)/\d s|_{s=0}$,  $[0,1]\ni t\mapsto \gamma_{[x,y]}$ is the locally unique null bicharacteristic of $H|_{s=0}$ connecting $x$ and $y$, and $\rho>0$ is a smooth  function. 
\end{enumerate}
\end{theorem}

\begin{proof}
We follow the proof of Theorem~2.1 in \cite{S-Lorentzian-scattering}. By the non-conjugacy assumption \r{A2}, the image of $\mathcal{H}=H^{-1}(\{0\})$  under the exponential map corresponding to each fixed $x$ is a codimension one hypersurface of $M$ near $y_0$. 
The image of the radial ray $\xi\mapsto t\xi$ originating from $x_0$, with $t$ close to $1$, is transversal to $V$ by \r{A1} near $y_0$. This shows that the intersection of that hypersurface and $V$ is a codimension one submanifold of $V$. We can include the initial point $x\in U$ as a parameter in those arguments as well. This proves \ref{Linearization_0_itema}. 

Consider \ref{Linearization_0_itemb}. It is enough to prove it for one choice of $r$, see also the argument at the end of the proof of the theorem. 
We chose a defining function $r$  using the notion of energy defined in Lemma~\ref{lemma_V}. Let 
\begin{equation}  \label{T21a}
E(\gamma) = \int_0^1H(\gamma(t) )\,\d t
\end{equation}  
be the energy of the Hamiltonian  curve $[0,1]\ni t\mapsto \gamma(t)= (x(t),\xi(t)) \in T^*M\set$.  The integrand, independent of $t$, vanishes on null bicharacteristics only. With the parameterization $\gamma_{[x,y]} (t )$ by initial and end-points as in (c), we get $E( \gamma_{[x,y]} )=0$ on $U\times V$ if and only if $(x,y)\in \Sigma$. We set 
\[
r(x,y) = E( \gamma_{[x,y]} ).
\]
We claim that $r$ is a defining function of $\Sigma$. We prove below that its differential does not vanish there. 

The variational Lemma~\ref{lemma_V} remains valid in the zero energy case. 
We get, arguing as in \r{T2},
 \begin{equation}  \label{T2a}
\p_s\big|_{s=0}r=  \p_s E(\gamma(\,\cdot\,;s);s)\big|_{s=0} =  \langle \eta^0,w\rangle- \langle \xi^0,v\rangle
 \end{equation}  
with $\xi^0$, $\eta^0$, $v$ and $w$ as in that proof. Since this holds for every $v\in T_{x_0}U$, $w\in T_{y_0}V$,  we get $r_x=-\xi^0$, $r_y = \eta^0$ at $s=0$. Those covectors are normalized by the requirement $t\in [0,1]$. Removing that normalization places the factor $\lambda>0$ in \r{Sigma}.  

In order to prove that $\d r\not=0$ on $\Sigma$; assume that it does not hold. Then $\xi^0{}'=0$ at $x_0$, and also $\eta^0{}'=0$ at $y_0$. This shows that $\xi^0\not=0$  is conormal to $U$ at $x_0$, similarly, $\eta^0\not=0$  is conormal to $V$ at $y_0$. We can  replace $\d \rho$ by $\xi^{0}$  in \r{A1}, and respectively by $\eta^0$, but then we get a contradiction with Euler's identity implying $\xi\cdot H_\xi=0$. Note that we showed a bit more than $(\xi^0{}',\eta^0{}')\not=0$; we showed that each component does not vanish.  This is consistent with the expectation that $C$ would locally map $T^*U\set$ to $T^*V\set$. This proves part \ref{Linearization_0_itemb}. 

Consider \ref{Linearization_0_itemc} now. With 
\[
r(x,y;s) \coloneq E(\gamma(\,\cdot\,;s);s) \coloneq \int_0^1H(x(t;s),\xi(t;s);s)\,\d t, 
\]
   write
\begin{align}
\frac{\d}{\d s}\Big|_{s=0} r(x,y;s) &= \int_0^1H_s(x(t;0),\xi(t;0);0)\,\d t + \frac{\d}{\d s}\Big|_{s=0} E(\gamma(\,\cdot\,;s))\\
&= - \int_0^1H_s(x(t;0),\xi(t;0);0)\,\d t,
\end{align}
yielding \r{dr} with $\rho=1$ there. Finally, any other defining function of $\Sigma(s)$ is of the kind $\alpha(x,y;s) r(x,y;s)$, $\alpha>0$ on $\Sigma(s)$. Differentiating with respect to $s$ at $s=0$, this brings the additional non-vanishing factor $\rho(x,y) = \alpha(x,y,0)$ above, as claimed. 
\end{proof}

\subsection{The Hamiltonian light ray transform}

The previous section leads us naturally to the ``Hamiltonian light ray transform''
\begin{equation}   \label{L1}
Lf(\gamma) =  \int f\circ\gamma(t)\,\d t = \int f(x(t),\xi(t))\, \d t, \quad f\in C^\infty(T^*M\set)
\end{equation}
over null rays ($H=0$). We assume that $f$ is homogeneous of  degree $m$ with some $m$ because of the freedom we have with rescaling, and in view of Lemma~\ref{lemma_hom}. Recall that $\Gamma_0\subset H^{-1}(0)$ is a conic codimension one set. Then $f$ has to be defined there as well, and it is natural for it to be homogeneous. 

 Note that \eqref{dr} determines $Lf$ up to a smooth elliptic multiplier only. That non-uniqueness is caused by the freedom to multiply $H$ by such a factor without changing $S_0^H$. On the other hand, injectivity of $L$ and a possible microlocal invertibility are unaffected by such a factor.

\begin{theorem}
$Lf=0$  for every null bicharacteristic contained in $\Gamma_0$ if and only if $f=X_H\phi$ on $\Gamma_0$ with some $\phi(x,\xi)$ defined there,  homogeneous of order $m-1$ and vanishing for $x\in\bo$.  This is equivalent to $f= X_H\phi+\mu H$ near $\Gamma_0$ with $\phi$ any smooth homogeneous extension of $\phi$ near $\Gamma_0$, and some $\mu(x,\xi)$ homogeneous of order $m-2$. 
\end{theorem}

\begin{proof}
The proof is as that of Theorem~\ref{thm_X1}. The homogeneity of $\phi$ follows by construction. 
For the second part, take any extension of $\phi$ as above. Then $f-X_H\phi =0 $ on $\Gamma_0$, thus  $f= X_H\phi+\mu H$ with some $\mu$ since $H$ is a defining function of $H^{-1}(\{0\})$. Then away from $\Gamma_0$, $\mu$ is uniquely determined from that equation, and in particular, it must be  homogeneous of order $m-2$. By continuity, this would be true on $\Gamma_0$ as well. 
Note that another extension of $\phi$ would produce another  $\mu$ having the same value on $\Gamma_0$. 
\end{proof}

The extensions of $\phi$ and $\mu$ do not need to be homogeneous outside $H^{-1}(0)$ but we keep them this way because we assume $f$ homogeneous as well. The equivalency statement above has to be considered in the homogeneous class.

\section{Rigidity of Finsler manifolds} \label{sec_Finsler}

In this section we apply the results of Section~\ref{sec_2} to prove a rigidity result for Finsler metrics. We assume that $M$  is a manifold with boundary equipped with a Finsler structure, that is, a smooth function $F:TM\set\to (0,\infty)$  whose restriction to the tangent space $T_xM$ is a Minkowski norm for each $x\in M$. This means by definition that $F(x,\cdot)$ is positively homogeneous of order one, and its Hessian\footnote{{The subscript $v$ here and throughout the section does not stand for differentiation.}} $g_v(x) \coloneq \frac12\partial_v^2 F^2(x,v)$ is a positive definite quadratic form for each $(x,v)\in TM\set$. We extend $F$ by continuity to the zero section for convenience. 
Note that the Finsler structure need not be reversible, i.e. invariant under the map $(x,v)\mapsto (x,-v)$, and we do not assume this.
A standard reference for Finsler geometry is \cite{BaoChernShen}.
A Riemannian metric $g$ induces a reversible Finsler structure upon setting  $F=\sqrt{g(v,v)}$. 
Inverse problems in Finsler geometry have attracted interest in the recent years, partly due to its relevance in the study of the propagation of elastic waves in anisotropic elasticity (see e.g.,~\cite{HoopIL2023, HoopIL2021, MonkonenRandersMetrics}).

The dual Finsler function $F^*(x,\xi)$ on $T^*M$ is defined on each $T_x^*M$, $x\in M$, as the maximum of $\xi(v)$ for $F(x,v)=1$. 
It turns out that $F^*(x,\cdot)$ is a Minkowski norm on each $T^*_xM$, called a co-Finsler norm, see \cite[Sec. 14.8]{BaoChernShen},
in particular $\frac12 (F^*(x,\cdot ))^2$ has a positive definite Hessian away from the zero section.
The Legendre transform $\mathcal{L}_F:TM\set\to T^*M\set$ with respect to $\frac12F^{2}$ is defined pointwise for each $x$, mapping $T_xM\setminus\{0\}$ to $T_x^*M\setminus\{0\}$ by $\L_F(x,v)=\partial_v(\frac{1}2F^2)=F\partial_v F=  g_v(x)(v,\cdot)$, where we used Euler's identity. It is a diffeomorphism, in general not linear in the fibers, and satisfies
\begin{equation}\label{eq:Legendre}
	F=F^*\circ \mathcal{L}_{F}.
\end{equation}  
Moreover, $(\L_{F})^{-1}=\L_{F^*}$.
The Finsler geodesic flow $\Phi_{L}^{t}$ on $TM$ is the flow determined by the Euler–Lagrange equations for $L=\frac12 F^2$. The co-Finsler flow on $T^*M$ is the Hamiltonian flow of $H\coloneq \frac12 (F^*)^2$. The corresponding flows are conjugated by $\mathcal{L}_F$:
\begin{equation}\label{eq:conjugation}
	\Phi^t_H\circ \mathcal{L}_F = \mathcal{L}_F\circ \Phi^t_L.
\end{equation}
(This is a general fact about Lagrangians with positive definite Hessians, see e.g. \cite[Theorem 20.10]{CannasDaSilva}.)
The Hamiltonians we obtain from Finsler metrics are homogeneous of order two, and have convex Hessians. We will call them convex (homogeneous of order two) Hamiltonians. Conversely, given any such Hamiltonian, we can apply the Legendre transform related to it to get a Finsler metric on $M$. This gives us an one-to-one correspondence between Finsler metrics and strictly convex, homogeneous of order two Hamiltonians, which we exploit. 

In contrast to Section~\ref{sec_2}, we work in the tangent bundle.
We denote by $SM=\{(z,v)\in M:F(v)=1\}$ the unit sphere bundle with respect to $F$, which is a smooth manifold with boundary $\partial SM$.
Let $U$ be an open subset of $\bo$ (we can always think of $M$ being included in a larger smooth Finsler manifold, if needed), and let $\mathsf U$
be an open subset of $SM$ over $U$ consisting of inward pointing unit vectors\footnote{This means that for any boundary defining function $\rho$ 
for $M$ we have $ \d \rho(v)>0$ for all $(x,v)\in \mathsf{U}$.}.
 (As before, we only treat the inward pointing case, since the outward poining one is similar.) 
Note that unlike here, in Section~\ref{sec_2}, $\mathcal{U}$ was a subset of $T^*U$.
We define $\mathcal{G}\subset TM$ similarly to $\Gamma$ in \eqref{eq:Gamma} (so $\mathcal{G}$ is conic) and we assume that all geodesics emanating from initial velocities in $\mathsf{U}$ intersect $\partial M$ for the first time transversely, in an open set $V\subset \partial M$. 
By \eqref{eq:Legendre}, unit codirections in the sense of section~\ref{sec_2} correspond to unit directions in the Finsler sense. 
Here we will make sense of the scattering relation as a map 
\begin{equation}
	\hat {\mathsf{S}}^{F}:\mathsf{U}\to SM\big|_V,
\end{equation} similarly to \eqref{a2}, although here we only need it on unit length vectors. 
This is also the approach in other works studying versions of the scattering relation on Finsler manifolds, see \cite{HoopIL2021}.   The travel time $\ell: \mathsf{U}\to \R$ is defined similarly to the Hamiltonian case.

In Section \ref{sec_2}, the scattering relation was defined between subsets of the cotangent to the boundary, and we would like to use the results obtained there.
We write $\partial_\pm SM = \{(z,v)\in SM : \pm \d \rho (v)>0\}$ for the incoming/outgoing portion of the boundary, where $\rho$ is any boundary defining function.
We have a natural map 
\begin{equation}\label{eq:R_map}
	\mathcal{R}_{\pm }:\partial_\pm SM\to T^*\partial M,\quad v\mapsto \iota^{*}\circ \L_F(v).
\end{equation}
In coordinates near a boundary point as usual, it satisfies $\xi_\alpha = \p_{v^\alpha}  (\frac12 F^2(x,v) )= \sum_{1\leq j\leq n}g_{\alpha j}(v)v^j $, $\alpha\le n-1$.
From this expression one can see, for instance, that \eqref{eq:R_map} depends on the values of $F$ only in a neighborhood of $v$ -- in principle, $F$ does not even need to be defined elsewhere. We assume that $F$ is defined globally mainly for notational convenience.

\begin{lemma}\label{lm:R_map}The map $\mathcal{R}_{\pm }$ is a diffeomorphism onto its image.
\end{lemma}
\begin{proof}
Recall that  $F^{*}$ is a Minkowski norm and $\L_{F}:TM\set \to T^*M\set $ is a diffeomorphism.
By \eqref{eq:Legendre}, it induces a diffeomorphism $SM\to S^*M$, where $S^*M$ is the unit cosphere bundle of $H=\frac12(F^{*})^{2}$ as in \eqref{norm}.
We denote the spaces of incoming/outgoing covectors by
\[\partial_\pm S^{*}M=\{(x,\xi)\in \partial S^{*}M: \d \pi X_H\in \partial_\pm SM\}.\]
Since $\d\pi X_H(\xi)=H_\xi(\xi)=\L _{F^*}(\xi)=
(\L_F)^{-1}(\xi)$, the property of being incoming/outgoing is preserved by $\L_F$ and we obtain diffeomorphisms
 $\L_F:\partial _\pm SM\to \partial_\pm S^*M$. 
It suffices to show that 
\begin{equation}\label{eq:rest}
	\iota^*:\partial _\pm S^*M\to  T^*\partial M
\end{equation} is a diffeomorphism onto its image.

Choose coordinates $(x',x_n)$ as usual near a boundary point, so that  $(x,\xi')=\iota^{*}(x,\xi)$ for some $(x,\xi)\in S^*_xM$, $x\in \partial M$, where $\xi=(\xi',\xi_n)$.
To show that \eqref{eq:rest} is injective, it suffices to show that for given $(x,\xi')\in \iota^{*}(\partial_\pm S^{*}M)$ there exists exactly one $\xi_n$  such that $ H(x,\xi',\xi_n)=1/2$ and $ H_\xi(x,\xi',\xi_n)\in \partial _\pm SM$ (note that $\xi_n$ could have any sign).
Denote the unit co-ball bundle by
\[B^{*}M=\{(x,\xi)\in  T^{*}M:H(	x,\xi)\leq 1/2\}.\]
As a consequence of the triangle inequality for Minkowski norms (see \cite[Theorem 1.2.2]{BaoChernShen}), the set $B_x^*M$ is strictly convex for every $x\in M$, in the sense that any line segment in $B_x^*M$ joining two distinct points is contained in its interior.
For fixed $(x,\xi')$, consider the function  $t\overset{f}{\mapsto} H(x,\xi',t)$, which satisfies $f''>0$ and $f\overset{|t|\to \infty}\to \infty $ by convexity and homogeneity of $H$.
The $\xi_n$'s in $\iota^{*}(x,\xi',\xi_n)=(x,\xi')$ correspond exactly to the solutions of $f(t)=1$.
There is at least one  such solution since $ (x,\xi)\in \partial _\pm SM$, and it satisfies $f'(t)=H_{\xi_n}(x,\xi,t)\neq 0$, so there has to be exactly one more where $f'$ has the opposite sign, by the aforementioned properties of $f$.
Therefore, \eqref{eq:rest} is injective.
It is a local diffeomorphism by the condition $H_\xi\neq0$, which allows us to solve the equation $H(x,\xi,\xi_n)$ smoothly for $\xi_n$ in a neighborhood of $(x^0,\xi_0)\in \partial_\pm S^*M$.
\end{proof}

\begin{remark}\label{rmk:alt}
	There is another natural map $\partial _\pm SM\to T^*\partial M$, which can be constructed as follows. As discussed  in \cite[p. 27-28]{ShenLecturesFinsler}, there exists a unique inward/outward pointing unit normal vector field $\nu_\pm$ to $\partial M$, in the sense that $g_{\pm \nu}(\pm \nu,w)=w$ for all $w\in T\partial M$ ($\nu_\pm$  are not necessarily linearly dependent). Hence we obtain well defined decompositions $TM=\R\nu _\pm\oplus T\partial M$, so that one can define projections $\mathcal{P}_{\pm }:\partial TM\to T\partial M$.
	Since $F$ induces a Finsler structure $\hat{F}$ on $\partial M$, we can define a map $\partial _\pm SM\to T^*\partial M$ by $\L_{\hat F}\circ \mathcal{P}_\pm $. If $F$ comes from a Riemannian metric this map agrees with $\mathcal{R}_{\pm }$, though in general this need not be the case.
\end{remark}
	
Note that by Lemma \ref{lm:R_map}, an inward/outward poiting $v\in \p SM$ can be recovered  given the form $\langle \xi',w\rangle \coloneqq  g_v(v,w)$ for all $w\in T_x\bo$.
By \eqref{eq:conjugation}, we have $\hat{\mathsf{S}}^F=\L_F^{-1}\circ \hat{S}^H\circ \L_F$, recall \eqref{a2}.
Moreover, by Definition \ref{def_Ls} we can write on $\mathsf{U}$
\begin{equation}\label{eq:Scat_tang}
	\hat{\mathsf{S}}^F=\mathcal{R}^{-1	}_{-}\circ S\circ\mathcal{R}_+ .
\end{equation}
This is our reason for defining $\mathcal{R}_{\pm }$ as in \eqref{eq:R_map} instead of using the approach outlined in Remark~\ref{rmk:alt}.

\begin{remark}
	In the study of the lens rigidity problem in the Riemannian case one often defines the scattering map $\mathsf{S}:T\partial M\to T\partial M$, by taking projections and ``inverse projections'' of the outgoing and incoming data of $\hat{\mathsf{S}}$. This has the advantage of requiring knowledge of the unknown metric $g$ only on $T\partial M$, as opposed to $\partial TM$. In our case, defining such a projection as described in Remark \ref{rmk:alt} does \emph{not} interact well with the Hamiltonian formulation of the problem. To take the Hamiltonian approach, one could define $\mathsf{S}\coloneqq \L_{\hat{F}}^{-1} \circ \mathcal{R}\circ \hat{\mathsf{S}}\circ \mathcal{R}^{-1}\circ \L_{\hat{F}}$, however this seems somewhat artificial. Of course, in the Riemannian case the two approaches agree.
\end{remark}

Let $\tilde M$ be a second manifold with boundary such that $\partial M=\partial \tilde M$.
Write $H=\frac{1}{2}(F^{*})$ as before and $\Gamma=\L_F(\mathcal{G})$.
We denote by $\mathcal{K}_{\Gamma,H}$ the set of homogeneous canonical transformations $\kappa:\Gamma \to \kappa(\Gamma)\eqqcolon \tilde{\Gamma}\subset T^* \tilde M$ such that $\tilde H\coloneqq H\circ \kappa^{-1}$ is strictly convex on $\tilde{ \Gamma}$ (see Remark \ref{rmk:convex} below).
Since then $\tilde H=\frac12(\tilde F^{*})^{2}$ for a Finsler metric $\tilde F$, for $\kappa\in \mathcal{K}_{\mathcal G,H}$
we have the chain of maps
\begin{equation}\label{eq:Mathieu}
	(\mathcal{G},F) \overset {\mathcal{L}_F}{\longrightarrow}  (\Gamma,H) \overset {\kappa}{\longrightarrow}  (\tilde \Gamma,\tilde H) \overset {\mathcal{L}_{\tilde F^*}}{\longrightarrow}   (\tilde {\mathcal{G}},\tilde  F).
\end{equation}

\begin{lemma}\label{lm:conjugation}
Let $F$ be a Finsler metric on $M$ with $\mathcal{G}$ as before, and suppose $\kappa\in \mathcal{K}_{\mathcal G,H}$ is a canonical transformation satisfying \eqref{kappa_b}.
The composition 
$\phi \coloneqq  {\mathcal{L}}_{\tilde F^{*}}\circ\kappa\circ \mathcal{L}_F\colon \mathcal{G}\to \tilde {\mathcal{G}}$
of the maps in \eqref{eq:Mathieu} (using the notation there) satisfies $\hat {\mathsf{S}}^{\tilde F}\circ \phi=\phi\circ\hat{\mathsf{S}}^F$.
\end{lemma}
\begin{proof}
We have already seen that \eqref{kappa_b} implies $S^{H}={S}^{\tilde H}$. We have then 
\begin{align}
	\hat{\mathsf{S}}^{\tilde F}\circ \phi=&\tilde{\mathcal{R}}^{-1}_{-}\circ  S^{\tilde H}\circ \tilde {\mathcal{R}}_{+ }\circ \phi
	=\tilde{\mathcal{R}}_{- }^{-1}\circ  S^{H}\circ \tilde {\mathcal{R}}_{+ }\circ \phi
	 \overset{\eqref{eq:Scat_tang}}=\tilde{\mathcal{R}}^{-1}_{- }\circ\mathcal{R}_- \circ\hat{\mathsf{S}}^F\circ\mathcal{R}^{-1}_+ \circ \tilde {\mathcal{R}}_{+ }\circ \phi.\label{eq:intertw2}
\end{align}
Writing $\zeta_\pm$ for the incoming/outgoing inverse of $\iota^{*}$ as described in the proof of Lemma \ref{lm:R_map}, we have 
\begin{align}
	\mathcal{R}^{-1}_{\pm }&\circ \tilde {\mathcal{R}}_{\pm }\circ \phi=\L_{{F}}^{-1}\circ{\zeta}_\pm \circ \iota^{* }\circ \L_{\tilde{F}}\circ ({\mathcal{L}}_{\tilde F^{*}}\circ\kappa\circ \mathcal{L}_F)
	=\L_{{F}}^{-1}\circ{\zeta}_\pm \circ \iota^{* }\circ\kappa\circ \mathcal{L}_F\\ 
	&\overset{\eqref{kappa_b}}{=}\L_{{F}}^{-1}\circ{\zeta}_\pm \circ \iota^{* }\circ \mathcal{L}_F=\Id,
\end{align}
since ${\zeta}_\pm \circ \iota^{* }=\Id $ on $\partial _\pm S^*M$. Substituting in \eqref{eq:intertw2} we obtain the claim.
\end{proof}

We have the following analog of Theorem~\ref{thm_H}.

\begin{theorem}[Semi-global rigidity of Finsler manifolds] \label{thm_F}  The notations in the theorem are as defined in the present section.
\begin{enumerate}[leftmargin=.8 cm,label=(\alph*)]
	\item \label{Finsler:itema} (The direct problem.) 
Let $(M,F)$  be a Finsler manifold with boundary, with $\mathcal{G}$ as above. Let $\tilde M$ be another manifold with $\bo = \p \tilde M$, and let $\kappa \in \mathcal{K}_{\Gamma,H}$ be a canonical transformation from $\Gamma$ to $\tilde \Gamma$ 
satisfying the boundary condition $\kappa(x,\xi)=(x,\xi)$ on ${\Gamma}\big|_{\bo}$.  
 If  $\tilde H\coloneqq H\circ \kappa^{-1}=\frac12(\tilde F^{*})^{2} $, we have $\hat {\mathsf S}^{F}= \hat{\mathsf S}^{\tilde F }$ on $\mathsf{U}$ and $\tilde \ell=\ell$ for the corresponding travel times.

 \item \label{Finsler:itemb} (The inverse problem.) 
Let $(M,F)$ and $(\tilde M, \tilde F)$ be two Finsler manifolds with a common boundary and assume that $F=\tilde F$ on $TM\big|_{\bo}$. 
Assume  $\hat{\mathsf{S}}^{\tilde F}= \hat{\mathsf{S}}^F$, $\tilde \ell=\ell$ on $\mathcal{U}$. Then there exists a homogeneous canonical transformation $\kappa\in \mathcal{K}_{\Gamma,H}$ with $\kappa=\Id$ on $\Gamma\big|_{\partial M}$ such that $\phi^*\tilde F=F$ on $\mathcal{G}$,   where  $\phi=\L_{\tilde{F}}^{-1}\circ \kappa \circ \L_F$. In particular, $\phi=\Id $ on $\mathcal{G}\big|_{\partial M}$.
\end{enumerate}

\end{theorem}
 
\begin{proof}
The proof of \ref{Finsler:itema} follows immediately from Lemma \ref{lm:conjugation} upon noticing that now $\phi=\Id $ on $SM\big|_{\bo}.$
To prove \ref{Finsler:itemb}, note that by \eqref{eq:Scat_tang} and $\hat{\mathsf{S}}^{\tilde F}= \hat{\mathsf{S}}^F$, we have
$	\tilde {\mathcal{R}}_{- }^{-1}\circ{S}^{ \tilde H}\circ \tilde{\mathcal{R}}_{+ }= {\mathcal{R}}_{- }^{-1}\circ {S}^{H}\circ {\mathcal{R}}_+ .$
Since in our case $F=\tilde F$ on $TM\big|_{\partial M}$, $\mathcal{R}=\tilde{\mathcal{R}}$ and therefore $S^{H}={S}^{\tilde H}$. By Theorem \ref{thm_H}, there exists a homogeneous canonical transformation $\kappa $ satisfying \eqref{kappa_b} such that $H=\kappa^{*}\tilde H$ on $\Gamma$. 
The requirement $\iota^{*}\circ \kappa=\iota^{*}$ together with $F=\tilde F$ on $T M\big|_{\partial M}$, which implies
 $H=\tilde H$ on $T^{*} M\big|_{\partial M}$, yield that $\kappa=\Id $ on $\Gamma
 \big|_{\partial M}$.
The requirement $\kappa\in \mathcal{K}_{\Gamma,H}$ is satisfied automatically, since we know that $\tilde H=\frac12(\tilde {F}^{*})^{2}$ is a strictly convex Hamiltonian which is homogeneous of degree 2.
\end{proof}

\begin{remark}\label{rmk:convex}
Theorem~\ref{thm_F} poses the following question: given a strictly convex homogeneous of degree two Hamiltonian $H$ defined on $\Gamma\subset T^*M$,  can we characterize the set $\mathcal{K}_{\Gamma,H}$? Any injective local diffeomorphism on $\pi(\Gamma)$ lifts to an element of $\mathcal{K}_{\mathcal G,H}$ (and in fact any such canonical transformation preserves the convexity of \emph{every} convex Hamiltonian).
These can be perturbed to elements of $\mathcal{K}_{\mathcal G,H}$ that do not arise as lifts of diffeomorphisms.
To see this, pick such a $\kappa_0$ and a point $\zeta_0=(x_0,\xi^{0})\in T^*M\set$ and write $(y_0,\eta^0)=\kappa_0(x_0,\xi_0)$. Then by \cite[Prop. 25.3.3]{Hormander4} there exist coordinates $y$ in a neighborhood of $y_{0}$ and a homogeneous generating function $\phi(x,\eta)$, where $U$ is a neighborhood of $x_0$, with $\partial^{2}_{x\eta}\phi\neq 0$ such that in a neighborhood $U$ of $(x_0,\xi^{0},y_0,\eta^{0})$ we have $\mathrm{graph}{(\kappa_0)}=(x,\d_x\phi,\d _\eta \phi,\eta)$.
Since $\kappa_0$ is the lift of a diffeomorphism, there exists an invertible linear map such that $\eta_{j}=A_{j}^{i}(x)\p_{x^{i}}\phi$, in particular $\phi$ is a degree 1 polynomial in $\eta$.
Slightly perturbing $\phi$ in a compact set  while preserving its homogeneity but breaking the linear dependence on $\eta$, we obtain the graph of a canonical transformation in $\mathcal{K}_{\mathcal G,H}$ that does not originate from a diffeomorphism. See also Remark \ref{rem_can}.
\end{remark}

\appendix 

\section{Motivation: the Hyperbolic DN map} \label{sec_App}
We review some microlocal constructions motivating our definition of the scattering relation $S_\lambda$ and the travel times $\ell_\lambda$ in section~\ref{sec_2}, related to a positive energy level; and $S_0$ related to a zero energy level. We show that those quantities appear naturally in boundary observations. Instead of studying the Dirichlet-to-Neumann (DN)  map, we study the Dirichlet-to-Dirichlet (DD) microlocal map but the former can be obtained with the same construction, see Remark~\ref{rem_DN}. 

\subsection{Hyperbolic \texorpdfstring{\PDO}{PsiDO}\ equations with time-independent coefficients} \label{sec_A1}
Consider the pseudo-differential equation:
\begin{equation}   \label{a1}
(D_t^2-q(x,D))u\in C^\infty,
\end{equation}
where $q(x,\xi)\sim q_2(x,\xi)+ q_1(x,\xi)+\dots$ is a classical symbol of order two, i.e., $q_k$ is assumed to be homogeneous of order $k$ in $\xi$.  Assume $q_2>0$. 
The principal symbol $q_2$ is invariantly defined on $T^*M\set$. To define the subprincipal symbol  invariantly, we need to rewrite $q(x,D)$ as acting on half-densities first, and then the subprincipal symbol would be $q_1+(\i/2)\p_{x^j} \p_{\xi_j} q_2$, see \cite[Theorem~18.1.33]{Hormander3} and \cite{OSSU-principal}. 
The r.h.s.\ in \r{a1} can be assumed to be microlocally smooth only, near the bicharacteristics of interest, but that can be also done a posteriori, with microlocal cutoffs. The same remark applies to the smoothness assumptions below.

The propagation of singularities theorem relates the singularities to the Hamiltonian $p=-\frac12\tau^2+   H(x,\xi)$ in time-space (the factor $-1/2$ there is just for convenience), where $H=\frac12 q_2$. 
The Hamiltonian system reads
\begin{equation}  \label{App1}
\frac{\d}{\d s} t= -\tau, \quad \frac{\d}{\d s} x=   H_\xi, \quad  \frac{\d}{\d s}\tau =0, \quad  \frac{\d}{\d s}\xi = -   H_x
\end{equation}  
with some initial conditions. Therefore, $\tau=\text{const.}$, and we want to stay on zero energy level for $p$, implying $H(x,\xi)=\frac12 \tau^2$. The first equation implies   $t=t_0 -\tau s$. In other words,
\begin{equation}   \label{H_sol}
t=t_0-\tau s, \quad \tau=\text{const.}, \quad (x,\xi) = \Phi^s(x_0,\xi^0)\quad \text{on $H(x,\xi)= \frac12\tau^2$},
\end{equation}
where $\Phi^t$ is the Hamiltonian flow of $H$. 
We see that the bicharacteristic is future/past propagating as $s$ increases, when $\mp\tau>0$. On the other hand, $s$ has no intrinsic meaning in the propagation of singularities theorem, and we can replace it with $t$. Then  for $\eta: = -\xi/\tau$ (unit in our terminology), $\tau<0$,  we get 
\[
 \frac{\d}{\d t} x=   H_\xi(x,\eta), \quad     \frac{\d}{\d t}\eta = -   H_x(x,\eta).
\]
When $\tau>0$, we set $\eta=\xi/\tau$. Then  $(-\tau)^{-1}H_\xi(x,\xi)= - H_\xi(x,\eta)$, $-\tau^{-2}H_x(x,\xi)=  - H_x(x,\eta)$.  In this case, the Hamiltonian system is as above but with $H$ replaced by $-H$. We can set $\tilde t=-t$ now, and get the same system. The bottom line is that for $\tau<0$, the unit codirection in the effective Hamiltonian would be $\eta=-\xi/\tau$, while when $\tau>0$, we have $\eta=\xi/\tau$ propagating back in time. Thus our setup covers both cases. Note that we described how we can treat $t$ as a parameter (time) instead of one of the ``spatial'' variables. If we want to work in spacetime, we should stay with  \r{App1}. 

\begin{example}
Take the simplest case $H=|\xi|^2/2$  on $T^*\R^n$. Then the speed is  $|\dot x|= |H_\xi|=|\xi|$. We are free to rescale $s$ by any constant (or even by functions), which is equivalent to multiplying $p$ by  elliptic factors. This does not change the propagation of the singularities statement but the new zero bicharacteristics are such for a new Hamiltonian: $\mu p= \mu(-\tau^2/2+H)$ instead of what we have, $\mu=1$.   In this constant coefficient case, 
\[
t = t_0 -\tau s, \quad x=x_0 +  s\xi 
\]
with $\xi$, $\tau$ constant, $\tau^2=|\xi|^2\not=0$.  Setting $t_0=0$, we get 
\[
x= x_0- t\xi/\tau.
\]
Set $\eta= -\xi/\tau$. Then $\eta$ is unit, and in fact 
\begin{equation}   \label{a3c}
x=x_0+t\eta^0,\quad \eta=\eta^0
\end{equation} 
is a Hamiltonian curve for $H=|\eta|^2/2$ but now the energy level is  $H=1/2$. Note that the sign of $\tau$ did not matter above but the Hamiltonian $H$ in this case is an even function of $ \xi$.

Assume now that $M=\{(x',x^n)\in \R^n:x^n>0\}$. Given a singularity $(t,x,\tau,\xi')$ of boundary data $f$ on $\R\times\bo$, with the prime standing for a restriction, we get two $\xi$'s with that restriction: $\xi_{\pm} = (\xi',\pm \xi_n)$ with $\xi_n \coloneq (\tau^2-|\xi'|^2)^{1/2}$ assuming $|\xi'|<|\tau|$ (a  timelike singularity). The first one points into $M$, and the second one points outside of it. If $\tau<0$, they both propagate to the future, meaning $t$ increases with $s$, with $(\tau,\xi_\pm)$ propagating in/out $M$. When $\tau>0$, they are both past propagating. On the other hand, when $s$ \textit{decreases}, they propagate to the future. 

If we are just given $H$ in this example, and we want to connect it to \r{a1}, we start with \r{a3c} with $\eta^0$ unit. Restricting it to $T\bo$, we have $\eta^0{}'$ with $|\eta^0{}'|<1$, and there are two covectors $\eta_\pm$ with that restriction. The first one, $\eta_+$ corresponds to a propagation into $M$ for $t>0$. The second one, $\eta_-$ propagates into $M$ for $t<0$. 
\end{example}

Since $H>0$, we can factorize $\frac12\tau^2-H= \frac12( \tau-\sqrt{2H})(  \tau+\sqrt{2H})$. Depending on the sign of $\tau$, one of the factors is elliptic along the zero bicharacteristics of the other. This shows that it is enough to analyze just one of the factors microlocally on $H(x,\xi)= \frac12\tau^2$.

\subsection{Equations of real principal type} \label{sec_A2}
We consider the more general case now, of which section~\ref{sec_A1} is a special case. We return to the partial case in section~\ref{sec_A3}. 
Let $p$ be of real principal type as in section~\ref{sec_zero_1}. We assume the order to be $m\in \mathbb{Z}$. 
The following theorem is essentially known, see, e.g., \cite{St-Yang-DN} for a Lorentzian version, and directly related to the propagation of singularities theorem. 

\begin{theorem} \label{thm_L1} \ 
In the notations of section~\ref{sec_zero_1}, let $f\in \mathcal{E'}(U)$, $\WF(f)\in \mathcal{U}$. Let $u$ solve $p(x,D)u\in C^\infty$ in a neighborhood of $\Gamma_0$, and $u|_U = f$ mod $C^\infty$. Then the map
\[
\Lambda\colon f\mapsto u|_{V}
\]
is an elliptic FIO of order zero associated with the scattering relation $S_0$ of $p_m$ at zero energy level.
\end{theorem}

\begin{proof}[Proof (Sketch)]
In local coordinates in which $U=\{x^n=0\}$, we are looking for a parametrix of the form
\[
u(x) = (2\pi)^{-n+1} \int e^{\i \phi(x,\xi')} a(x,\xi) \hat f(\xi') \, \d\xi',
\]
where $a~\sim a_0+a_1+\dots$ is a classical amplitude of order zero, and $\phi$ solves the eikonal equation
\begin{equation}  \label{H_eik}
H(x,\d_x\phi) = 0, \quad \phi|_{x^n=0} = x'\cdot\xi'.
\end{equation}  
We solve it for $(x,\xi')$ in a conic neighborhood of $(x_0,\xi^0{}')$. The method of characteristics, see \cite{Evans}, says that we determine $\d_x\phi(x,\xi')$ for $x^n=0$ first by solving $H(x,\d_x \phi) = 0$ for $\p_{x^n}\phi$ there. This algebraic equation may have more than one locally smooth solution; we pick one of them. In our case, it is the one giving us $\p_{x^n}\phi=\xi_n$, see \r{zeta}. Then we pass to the $\xi'$ dependent coordinates $(x',s) \mapsto x(s) = \gamma_{(x',0),\xi(x')} (s)$, where $\gamma$ is the projection of the zero bicharacteristic on the base.  In those coordinates, $\phi=x'\cdot\xi'$. Then it happens that $\d_x \phi= \xi(s)$, where the latter is the fiber component of that bicharacteristic. To apply $p(x,D)$ to $u$, we use the asymptotic expansion of a \PDO\ applied to $e^{\i \rho\phi}a$, as $\rho\to\infty$. In this case, $\rho = |\xi'|$, see \cite[Theorem~VI.3.1]{Treves2}. The leading terms yield the eikonal equation \r{H_eik}. The second term, see (3.32) and (3.34)  in \cite{Taylor-book2}, implies the first transport equation:
\[
Ta_0= 0, \quad a_0|_{x^n=0}=1, 
\]
where
\[
T \coloneq \sum _{|\alpha|=1} (\p_\xi^\alpha p_m) (x,\p_x\phi) D_x^\alpha + p_{m-1}(x, \p_x\phi) + \i \sum_{|\alpha|=2} (\p_\xi^\alpha  p_m) (x,\p_x\phi)\frac1{\alpha!} D_x^\alpha\phi.
\]

The transport equation $Ta_0=0$ is an ODE along each zero bicharacteristic, projected to the base, as the leading term indicates. The next two terms are potential-like and they can be handled by an integrating factor technique. In particular, $a_0$ never vanishes. To determine $a_j$ for $j\ge1$, we solve an inhomogeneous transport equation, with the same $T$ but with a right-hand side determined by the previous steps. A similar construction is presented in \cite[Ch.~VIII]{Taylor-book0}, for example, by introducing an auxiliary time variable.  

Assume for a moment that the eikonal equation is solvable all the way to $V$, and a bit behind it. Then $\Lambda$ is a formal (at this point) FIO, microlocally from $\mathcal U$ to $\mathcal V$ with a Lagrangian parameterized by the phase $\psi\coloneqq  \phi(x',\xi') - y'\cdot\xi'$, $x\in V$, $y'\in U$, where $x'$ is $x$ restricted to $V$ in coordinates in which $V=\{x^n=0\}$ (note that here we swap the roles of $x$ and $y$ to make our notation consistent with \cite{Taylor-book0}). The associated canonical relation, see, e.g., \cite[VIII.5]{Taylor-book0}, is 
\begin{equation}\label{phi_c}
(\p_{\xi'}\phi, \xi') \mapsto (x',\p_{x'}\phi).
\end{equation} 
We claim (and this is known, in principle) that the above  reduces to 
\[
(y', \xi') \mapsto (x',\p_{x'}\phi(x',\xi')),
\]
where $y'$ is such that the bicharacteristic from $(y',0)\in U$  with co-direction $\xi$ having projection $\xi'$ hits $(x',0)$; then $\p_{x' }\phi(x',\xi')$ would be the co-direction there. This is exactly $S_0$.

To prove the claim above, let $(x,\eta)=( x(t,y', \xi'), \eta(t,y', \xi'))$ be that zero bicharacteristic. Consider $F(t)\coloneqq  \phi_{\xi'}(x(t,y', \xi'), \xi')$. Then $\dot F^\alpha(t) = \dot x^i (\p_{x^i}\p_{\xi_\alpha}\phi)= H_{\xi_i} (x,\eta) (\p_{x^i}\p_{\xi_\alpha}\phi )(x(t,y', \xi'), \xi') $. The latter vanishes as it follows by differentiating the Hamilton-Jacobi equation \r{H_eik} with respect to $\xi'$. Then $F(t)=F(0)=y'$. We refer to \cite[p.~322]{Treves2} for a similar argument. 

The construction above is local only, and may not extend all the way to $V$. On the other hand, it works in the same way starting from a small hypersurface $U_1$ intersecting $\gamma_0$ transversally, so that \r{A1} holds at that point, which can be achieved by choosing $U_1$ suitably. We can choose a partition of unity along $\gamma_0$ to write $\Lambda$ as a composition of a finite number of FIOs having a canonical relation the composition of the corresponding scattering relations, which is $S_0$, see also \cite{St-Yang-DN}.
 \end{proof}

\begin{remark}\label{rem_DN}
One can study the Dirichlet-to-Neumann (DN) map in Theorem~\ref{thm_L1} instead, as it is often done in the literature. The treatment is very similar, one needs to reflect the solution at $V$ with the Dirichlet boundary condition, and take the trace of the normal derivative at $V$, see, e.g., \cite{St-Yang-DN} for this construction in the wave equation case. The DN map is an elliptic FIO with the same canonical relation, of order one. 
\end{remark}

\begin{remark}\label{rem_can}
Relation \r{phi_c}, without the primes, allows us to construct a large set of canonical relations $\kappa$ satisfying the boundary condition in \r{kappa_b}, through a generating function $\phi(x,\xi)$  of kind III. It is enough to take $\phi(x,\xi)$ to be homogeneous of order one, with $\phi(x,\xi)=x'\cdot\xi'$ for $x=(x',0)\in U$, and $\xi$ in some conic neighborhood of some $\xi^0\not=0$; and similarly on $V$, with respect to the boundary coordinates there. As long as \r{phi_c} is a diffeomorphism from $\Gamma$ to $\tilde \Gamma$ (true, for example, when $\phi$ is small enough in some $C^k$), then we get a canonical transformation $\kappa$ with the desired properties. In particular, one can choose $\phi=x\cdot\xi$ away from some neighborhood of $U\cup V$; then $\kappa(x,\xi)$ would be identity for $x$ close to $U$ or $V$. Note that the $\kappa$'s obtained this way are generated by a diffeomorphism $F$ if and only if $\phi(x,\xi) = F^{-1}(x)\cdot\xi$. 
\end{remark} 	

\subsection{Back to hyperbolic \texorpdfstring{\PDO}{PsiDO}\ equations with time-independent coefficients} \label{sec_A3} 
We consider the special case of \r{a1} now. Then $p=\tau^2-q$ is homogeneous of order two, and dividing by $2$ for convenience, we have $p_2=\tau^2/2- H(x,\xi)$, $H=q_2/2$. The variables denoted by $x\in\R^n$ in section~\ref{sec_A2} are $(t,x)\in\R^{1+n}$ now. The analog to $(x_0,\xi^0)$ is a fixed point $(t_0,x_0,\tau^0,\xi^0) \in (T^*\R\times\mathcal{U})\set$. 
We preserve the notation $U$, $V$, $\mathcal{U}$, $\mathcal{V}$ to be as in section~\ref{sec_2}, related to $x$ only, keeping in mind that in the setup of section~\ref{sec_A2}, they get replaced by $\R\times U$, etc. We assume condition \r{A1} at $(t_0,x_0,\tau^0,\xi^0)$, and at the exit one, $(s_0,y_0,\tau^0,\eta^0)$. In particular, this fixes $\sgn(\tau_0)$. The sign of $s$ in \r{H_sol} depends on whether $\xi^0$ is incoming (then $s>0$) or outgoing (then $s<0$). Then $\Gamma_0$, related to $\tau^2/2-q_2/2$, takes the form
\[
\Gamma_0 = \left\{ (t,x,\tau,\xi)|\; (x,\xi)= \Phi^{-(t-\hat t)/\tau} (\hat x, \hat \xi) ,  \; (\hat x,\hat \xi)\in \mathcal{U}_{|\tau|} , \, \tau\in \sgn(\tau_0)\R,\,  \hat t\in \R  \right\},
\]
where the sign of $-(t-t_0)/\tau$ (which we called $s$ above), determining the sign of $t-\hat t$, depends on $\xi^{0}$ as explained. The range of the interval over which $t$ varies is determined by the requirement that the trajectory is between $U$ and $V$. 
 
\begin{theorem}
Let $f\in \mathcal{E'}(\R\times U)$, $\WF(f)\in (T^*\R\times \mathcal{U})\set $. Let $u$ solve $(D_t^2-q(x,D)) u\in C^\infty$ in  $\Gamma_0$ 
 and $u|_{\R\times U} = f$ mod $C^\infty$. Then the map
\[
\Lambda: f\mapsto u|_{\R\times V}
\]
is an elliptic FIO of order zero associated with the canonical relation 
\[
(t,\tau, (y',\eta'))  \mapsto (t+\sgn(-\tau) \ell(y',\eta'/|\tau| ), \tau ,S_{|\tau|} (y',\eta')), 
\]
where $S_{|\tau|}$ is the scattering relation of $H=\frac12 q$ at  energy level $\tau^2/2$. 
\end{theorem}

\begin{proof}
Set $p=-\frac12\tau^2 + H(x,\xi)$. The Hamiltonian flow then decouples as in \r{H_sol}. The canonical relation of $\Lambda$ then can be written as 
\[
(t,x', \tau,\xi')  \mapsto (t-\tau \ell_{|\tau|}(x',\xi'),y',\tau , \eta'), \quad (y',\eta') = S_{|\tau|}(x',\xi'),
\]
where $\ell_{|\tau|}(x',\xi')$ is the value of $s$ at which (the base projection of) $\Phi^s((x',0),\xi)$ reaches $V$, with $\xi$ solving $H((x',0),\xi)= \tau^2/2$. We have $\ell_{|\tau|}(x',\xi') = |\tau|^{-1} \ell(x',\xi'/|\tau|) $, which completes the proof. 
\end{proof}

\bibliographystyle{abbrv}
\bibliography{references.bib}

\end{document}